\documentclass[11pt]{amsart}

\usepackage{verbatim, amssymb,hyperref}

\usepackage{color}
\usepackage{bbm}
\usepackage{graphicx}
\usepackage{subcaption}
\usepackage{todonotes}

\newcommand{\eps}{\varepsilon}

\newcommand{\einschraenkung}{\,\rule[-5pt]{0.4pt}{12pt}\,{}}

\newcommand{\cF}{\mathcal F}

\newcommand{\cO}{\mathcal O}

\newcommand{\sfrac}[2]{\mbox{$\frac{#1}{#2}$}}

\newcommand{\Id}{\operatorname{Id}}

\newcommand{\spec}{\operatorname{spec}}

\newcommand{\IV}{\mathbb V}

\newcommand{\IB}{\mathbb B}

\newcommand{\1}{1\hspace{-0.098cm}\mathrm{l}}

\renewcommand{\P}{{\mathbb P}}

\newcommand{\N}{{\mathbb N}}
\newcommand{\IA}{{\mathbb A}}

\newcommand{\IC}{{\mathbb C}}
\newcommand{\E}{{\mathbb E}}

\newcommand{\R}{{\mathbb R}}
\newcommand{\cT}{{\mathcal T}}
\newcommand{\cN}{{\mathcal N}}

\newcommand{\Hess}{\operatorname{Hess}}

\newcommand{\proj}{\operatorname{proj}}

\theoremstyle{plain}

\newtheorem{theorem}{Theorem}[section]
\newtheorem{proposition}[theorem]{Proposition}
\newtheorem{lemma}[theorem]{Lemma}
\newtheorem{corollary}[theorem]{Corollary}
\newtheorem{definition}[theorem]{Definition}

\theoremstyle{definition}

\title[Optimal Asymptotic Rates for (S)GD under the Local PL-Condition]%
{Optimal Asymptotic Rates for (Stochastic) Gradient Descent under the Local PL-Condition: A Geometric Approach}

\author[]
{Sebastian Kassing}
\address{Sebastian Kassing\\
	Department of Mathematics \& Informatics,  
  University of Wuppertal,
    42119 Wuppertal, Germany}
\email{kassing@uni-wuppertal.de}

\author[]
{Thomas Kruse}
\address{Thomas Kruse\\
	Department of Mathematics \& Informatics,  
  University of Wuppertal,
    42119 Wuppertal, Germany}
\email{tkruse@uni-wuppertal.de}

\keywords{Stochastic gradient descent, PL-inequality, rate of convergence, stable manifold}
\subjclass[2020]{Primary 90C26; Secondary 90C15, 62L20, 90C30}

\begin{document}

\begin{abstract} 
	Stochastic gradient descent (SGD) has been studied extensively over the past decades due to its simplicity and broad applicability in machine learning.
    In this work, we analyze the local behavior of gradient descent and stochastic gradient descent for minimizing $C^2$-functions that satisfy the Polyak-\L ojasiewicz (PL) inequality and under a multiplicative gradient noise model motivated by overparameterized neural networks. Using a geometric interpretation of the PL-condition, we prove a simple yet surprising fact: in this possibly non-convex setting, the asymptotic convergence rate of (S)GD matches the rate obtained for strongly convex quadratics.
\end{abstract}

\maketitle

\section{Introduction}
Many machine learning problems involve the minimization of a function $f: \mathbb{R}^d \to \mathbb{R}$ given as an expectation
\begin{align} \label{eq:loss}
f(\theta) = \mathbb{E}[\tilde f(\theta,X)],
\end{align}
where $X$ is a random variable representing the data distribution and $\tilde f$ is a loss function measuring the discrepancy between the prediction of the model with the parameters $\theta$ and the observed data.

For example, in supervised learning one is typically given a finite training data set $x_1,\dots,x_N$, consisting of  independent samples drawn from the distribution of $X$. 
A standard approach is therefore \emph{empirical risk minimization} (ERM), where the expected loss is approximated by the empirical average
\begin{align} \label{eq:empiricalloss}
f_N(\theta) = \frac{1}{N}\sum_{i=1}^N \tilde f(\theta,x_i).
\end{align}
The aim is then to find parameters $\theta \in \mathbb{R}^d$ that minimize this empirical risk $f_N$, which serves as an accessible surrogate for the true expected loss $f$ given by \eqref{eq:loss}. In many modern applications, the number of samples $N$ is very large, making the evaluation of the full gradient $\nabla f_N(\theta)$ computationally expensive. This motivates the use of stochastic optimization methods, most prominently stochastic gradient descent (SGD). Instead of computing the exact gradient, SGD uses a stochastic approximation of the gradient obtained from randomly sampled data points.

In this work, we restrict our attention to the optimization error, i.e., depending on the situation, the minimization of either $f$ or $f_N$, and ignore the statistical approximation error $f-f_N$. Since $f_N$ can also be written as an expectation, we will from now on denote the objective function by $f$ and adopt the definition in \eqref{eq:loss}. Next, we define stochastic gradient approximations to the gradient $\nabla f$.

Let $(X_{k,i})_{k\in\mathbb{N}, \; i=1,\dots,M}$ be i.i.d.\ samples that have the same distribution as $X$, where $M \in \N$ denotes the mini-batch size. Define a filtration $(\cF_{k})_{k \in \N_0} := (\sigma(\{X_{j,i}: j \le k, i \le M\}))_{k \in \N_0}$. At iteration $k \in \N$, the stochastic gradient is defined as the mini-batch average
$$
G_k(\theta_{k-1}) 
= \frac{1}{M}\sum_{i=1}^{M} \nabla \tilde f(\theta_{k-1},X_{k,i}).
$$
Note that $G$ satisfies $\mathbb{E}[G_k(\theta)] = \nabla f(\theta)$. 
Hence, the stochastic gradient can be decomposed into
$$
G_k(\theta_{k-1}) = \nabla f(\theta_{k-1}) + D_k,
$$
where
$$
D_k := \frac{1}{M}\sum_{i=1}^{M} \nabla \tilde f(\theta_{k-1},X_{i,k}) - \nabla f(\theta_{k-1})
$$
represents the stochastic gradient noise and satisfies $\mathbb{E}[D_k \mid \cF_{k-1}] = 0$.

With this example in mind, we give a general definition of the SGD scheme considered in this work.

\begin{definition}
Let $f: \R^d \to \R$ be a continuously differentiable function, and let $(\mathcal F_k)_{k \in \N_0}$ be a filtration. We call an adapted stochastic process $(\theta_k)_{k \in \N_0}$ that satisfies for all $k \in \N$
\begin{equation} \label{eq:SGD}
\theta_k = \theta_{k-1} - \gamma \big(\nabla f(\theta_{k-1}) + D_k\big),
\end{equation}
\emph{stochastic gradient descent scheme (SGD)}, 
where $\gamma>0$ denotes the step-size and $(D_k)_{k \in \N}$ is an adapted integrable process satisfying for all $k \in \N$ that
$$
    \E[D_k \mid \cF_{k-1}]=0.
$$
By setting $D_k \equiv 0$, the deterministic gradient descent scheme can also be expressed as in  \eqref{eq:SGD}.
\end{definition}

An extensive classical and recent literature examines the convergence properties of (stochastic) gradient descent under different structural assumptions on the objective function $f$ and the gradient noise $(D_k)_{k \in \N}$. We refer the reader, e.g., to \cite{robbins1951stochastic, ljung1992stochastic,  kushner2003stochastic,  absil2005convergence, benaim2006dynamics,  tadic2015convergence, bottou2018optimization, mertikopoulos2020almost,  dereich2021convergence, garrigos2023handbook} and the references therein.

In this work, we consider the setting in which $f$ satisfies a local Polyak-\L ojasiewicz (PL) inequality \cite{polyak1963gradient, lojasiewicz1963propriete} and the gradient noise vanishes at the minima. 
More precisely, we derive an asymptotic convergence rate for SGD under the assumption that the method finds a small neighborhood of a critical point at which a local PL-inequality holds and such that the conditional second moment of $D_k$ is bounded by the optimality gap. This assumption on the error term is often referred to as the overparameterized setting; see e.g. \cite{ma2018power, vaswani2019fast, liu2023aiming, gess2026exponential}. 

\begin{definition}\label{ass:PL}
    Let $0<\mu \le L$, $\sigma \ge 0$ and $f: \R^d \to \R$ be a continuously differentiable function. A local minimum $\theta^\ast \in \R^d$ of $f$ is called \emph{$(\mu,L)$-regular point} if there exists an open neighborhood $U \subset \R^d$ of $\theta^\ast$ such that $f\big|_U: U \to \R$
    \begin{enumerate}
        \item[(A1)] is two times continuously differentiable and minimized by $\theta^\ast$, i.e. $f(\theta) \ge f(\theta^\ast)$ for all $\theta \in U$,
		\item[(A2)] satisfies the PL-inequality for a constant $\mu >0$, i.e., 
		\begin{align} \label{eq:PL}
			\|\nabla f(\theta)\|^2 \ge 2 \mu (f(\theta)-f(\theta^\ast)), \quad \text{for all}\ \theta \in U, \text{ and } \tag{PL}
		\end{align}
		\item[(A3)] is $L$-smooth for a constant $L>0$, i.e., 
		$$
        \|\nabla f(\theta)-\nabla f(\theta')\|\le L\|\theta-\theta'\|,\quad\text{for all}\ \theta, \theta'\in U.
        $$
        \end{enumerate}
        If, additionally,       
        \begin{enumerate}
        \item[(A4)] for all $k \in \N$ one has almost surely on the event $\{\theta_{k-1} \in U\}$
        $$
            \mathbb{E}[D_k \mid \cF_{k-1}]=0 \quad \text{ and } \quad \E[\|D_k\|^2 \mid \mathcal F_{k-1}] \le \sigma (f(\theta_{k-1})-f(\theta^\ast)),
        $$
        $\theta^\ast$ is called \emph{$(\mu,L,\sigma)$-regular point}.
	\end{enumerate}	
\end{definition}

Note that all local minima in a small neighborhood of $\theta^\ast$ are also $(\mu,L,\sigma)$-regular points and have the same objective function value as $\theta^\ast$. We emphasize that our assumptions on the objective function are purely local and that strong convexity is not required, either globally or locally.

In the overparameterized regime, SGD is known to behave more like classical gradient descent, and convergence can be obtained without requiring vanishing step-sizes. Under global versions of assumptions (A2)-(A4), \cite{wojtowytsch2021stochastic} derives a convergence rate for SGD. More precisely, for step-sizes $\gamma < \frac{2\mu}{2\mu+\sigma} \frac{2}{L}$ one obtains linear convergence of $(f(\theta_k))_{k \in \N_0}$ with rate
\begin{equation}\label{eq:rate_wojtow}
    1-2\mu\gamma + \gamma^2 \frac{L(2\mu+\sigma)}{2}.
\end{equation}
In the noiseless case $\sigma=0$, this is consistent with the results in \cite{karimi2016linear} for gradient descent under the PL-condition (see also~\cite{vaswani2019fast, khaledbetter}).
Conceptually, this convergence rate results from combining a bound for the contractivity given by the PL-constant with a bound for the appearing error terms in the first-order Taylor approximation using the Lipschitz continuity of the gradient.

In the present work, we use ideas from differential geometry to further improve the asymptotic convergence rate for objective functions that locally satisfy the PL-inequality. We show that, although $f$ is possibly non-convex (even locally), the asymptotic convergence rate of SGD in the neighborhood of a $(\mu,L,\sigma)$-regular point coincides with the optimal convergence rate for the optimization of strongly convex quadratics with the same curvature restrictions. Although the statement may appear at first glance to be classical, it closes a genuine gap in the literature and is far from routine to prove.

Let us present a geometric intuition. 
In the case of optimizing a strongly convex quadratic using deterministic gradient descent (i.e., \(\sigma=0\)), the rate \(1-2\mu\gamma + \gamma^2 L\mu\) stated in \eqref{eq:rate_wojtow} for step-sizes $\gamma <  \frac{2}{L}$ can be improved to
$$
    \max\bigl(\lvert 1-\gamma\mu\rvert^2,\ \lvert 1-\gamma L\rvert^2\bigr).
$$
This improvement goes back to Polyak \cite{polyak1964some}.
The heuristic reason behind this is that in the quadratic case the problem can be decoupled into $d$ one-dimensional quadratic subproblems. For each one-dimensional quadratic, the PL-constant $\mu$ and the smoothness constant $L$ are both given by the second derivative at the critical point. 
Consequently, estimating the contraction by the smallest eigenvalue and the smoothness by the largest eigenvalue of the Hessian is overly pessimistic, since these bounds are attained in different directions rather than within the same one-dimensional subproblem. In this way, the flat and steep directions separate, leading to an improved estimate. Using a second-order Taylor approximation, this intuition can be carried over the the case of general strongly convex functions with Lipschitz continuous gradient; see \cite{polyak1964some}. We give a detailed account of the quadratic case including overparameterized gradient noise in Section~\ref{sec:quadratic}. 

The main contribution of this work is to show that a similar phenomenon occurs in the neighborhood of a $(\mu,L, \sigma)$-regular point. More precisely, we prove that local $C^2$-regularity together with a local PL-inequality is sufficient to asymptotically recover the optimal convergence rates of (S)GD on strongly convex quadratics, despite the possible non-convexity of $f$ and the potential non-uniqueness of the critical points. This main result is given by the following theorem.

\begin{theorem} \label{theo1}
	Let $0<\mu\le L$, $\sigma \ge 0$, and $f: \R^d \to \R$ be a continuously differentiable function. Let $(\theta_k)_{k\in\N_0}$ be the SGD scheme defined in \eqref{eq:SGD} with 
    \begin{equation}\label{eq:step_size_cond_main_thm}
         0<\gamma < \frac{2}{L+\frac{\sigma}{2}},
    \end{equation}
    and let $\theta^\ast \in \R^d$ be a $(\mu,L, \sigma)$-regular point.
    Then, for every $\delta >0$ there exists an open neighborhood $V$ of $\theta^\ast$ such that for the event $\IV := \bigcup_{k \in \N_0}\{\theta_k \in V\}$ the following holds: 
    there exists an event $\IC$ with $\P(\IC \cap \IV) \ge (1-\delta) \P(\IV)$ such that, on $\IC \cap \IV$, the SGD scheme $(\theta_k)_{k \in \N_0}$ almost surely converges to a (random) local minimum $\theta_\infty$ with $f(\theta_\infty) = f(\theta^\ast)$ and for every $\eps >0$ one has almost surely that
	\begin{align} \label{eq:rate1}
		\limsup_{k \to \infty} (m(\gamma)+\eps)^{-k}  \|\theta_k-\theta_\infty\|^2 =0\, ,
	\end{align}
	and
	\begin{align} \label{eq:mgamma}
		\limsup_{k \to \infty} (m(\gamma)+\eps)^{-k} (f(\theta_k)-f(\theta_\infty))=0\, ,
	\end{align}
	where
    $$
        m(\gamma):= \max\left((1 - \gamma \mu)^2+\frac{\gamma^2\sigma \mu}{2}, (1 - \gamma L)^2+\frac{\gamma^2\sigma L}{2}\right)
    $$
\end{theorem}

Setting $\sigma=0$, we immediately get the following result for the deterministic gradient descent scheme. 

\begin{corollary}
    Let $0<\mu\le L$, and $f: \R^d \to \R$ be a continuously differentiable function. Let $(\theta_k)_{k\in\N_0}$ be the gradient descent scheme given by
    $$
        \theta_k = \theta_{k-1} - \gamma \nabla f(\theta_{k-1}) \quad (k \in \N),
    $$
    with $0< \gamma < \frac 2 L$, and let $\theta^\ast \in \R^d$ be a $(\mu,L)$-regular point.
    Then, there exists an open neighborhood $V$ of $\theta^\ast$ such that the following statement is true: if there exists an $N \in \N_0$ with $\theta_N \in V$ then $(\theta_k)_{k \in \N_0}$ converges to a local minimum $\theta_\infty$ with $f(\theta_\infty) = f(\theta^\ast)$ and for every $\eps >0$ one has
	\begin{align*}
		\limsup_{k \to \infty} \Big(\max\big(|1-\gamma\mu|,|1-\gamma L|\big)+\eps\Big)^{-2k}  \|\theta_k-\theta_\infty\|^2 =0\, ,
	\end{align*}
	and
	\begin{align*}
		\limsup_{k \to \infty} \Big(\max\big(|1-\gamma\mu|,|1-\gamma L|\big)+\eps\Big)^{-2k}  (f(\theta_k)-f(\theta_\infty))=0\, .
	\end{align*}
\end{corollary}

The starting point for our analysis is the geometric perspective on the PL-inequality developed by Rebjock and Boumal in their remarkable work~\cite{rebjock2023fast}, which builds on an earlier result of Feehan~\cite{feehan2020morse}. They show that if $f \in C^2$ satisfies a local PL-inequality, the set of minima forms a $C^1$-manifold and the eigenvalues of the Hessian in the normal directions are lower bounded by the PL-constant. More precisely, their result is as follows.

\begin{proposition}[See Theorem~2.16 and Corollary~2.17 in~\cite{rebjock2023fast}] \label{thm:rebjock} 
    Let $\theta^\ast \in \R^d$ be a local minimum and $U\subset \R^d$ be a neighborhood of $\theta^\ast$ such that $f\big|_U$ is $C^k$, for a $k \ge 2$, and satisfies the PL-inequality \eqref{eq:PL} with constant $\mu >0$. 
    Then, there exists a neighborhood $U'$ of $\theta^\ast$ such that the set $\mathcal M :=\{\theta \in \R^d: \nabla f(\theta)=0\} \cap U'$ 
	is a $C^{k-1}$-submanifold of $\R^d$ and the Hessian satisfies
	\begin{align} \label{eq:MB}
        \ker(\Hess f(\theta^\ast)) = \cT_{\theta^\ast} \mathcal M \quad \text{ and } \quad 
		\langle v, \Hess f(\theta^\ast) \,  v \rangle \ge \mu \|v\|^2 \quad \text{ for all } v \in \cN_{\theta^\ast} \mathcal M.
	\end{align}
	Here, $\cT_{\theta^\ast}  \mathcal M$ denotes the tangent space of $\mathcal M$ at $x$ and $\cN_{\theta^\ast}  \mathcal M := (\cT_{\theta^\ast}  \mathcal M )^\perp$ denotes the normal space of $\mathcal M$ at ${\theta^\ast}$.
\end{proposition}

This is a strong result, since it allows flattening the set of local minima by a chart and treating the normal and tangent directions separately. In the normal directions, the minima are strongly contracting, whereas in the tangent directions there is no comparable restoring force. Nevertheless, locally, the objective value admits an upper bound that depends only on the distance to the manifold, implying that movements along the tangent directions are asymptotically negligible. Consequently, up to a second-order Taylor approximation, the problem factorizes into tangent directions, which are asymptotically irrelevant, and normal directions, where the objective is essentially a collection of one-dimensional strongly convex quadratics. 

The additional regularity for the set of minima established in \cite{rebjock2023fast} has recently been used in several works to derive fast convergence rates for several optimization methods, including stochastic gradient descent with classical noise assumptions \cite{fehrman2020convergence} and its Ruppert-Polyak average \cite{dereich2023central}, Polyak's heavy ball \cite{kassing2024polyak}, Nesterov accelerated gradient \cite{gupta2024nesterov, feng2026optimal}, and the regularized Newton method \cite{rebjock2023fast}. Moreover, it can be shown that SGD with non-degenerate gradient noise avoids manifolds consisting of strict saddle points; see \cite{mertikopoulos2020almost}. 
Under classical noise assumptions and using a decreasing step-size $\gamma_n \asymp 1/n$, the function value of SGD on objectives satisfying the PL-inequality converges with the rate $\cO(1/n)$. This is exactly the rate of convergence in the law of large numbers determined by the central limit theorem. Thus, in this situation, using the framework of decoupled quadratics can at most improve the rate of convergence by a multiplicative factor. In the overparameterized or deterministic setting considered in this work, however, one can prove an asymptotic speed-up in the exponential convergence rate.

It is worth mentioning that every coercive $C^2$-function that satisfies a \textbf{global} PL-inequality has a unique minimum and grows quadratically in the distance to this minimum, see \cite{criscitiello2025, nejma2025polyak}. This fact significantly restricts the class of functions that satisfy a global PL-inequality. Since the class of functions that contain a local minimum where a \textbf{local} PL-inequality holds is much larger, our result requires an independent geometric analysis of the convergence behavior of (S)GD.

The remainder of the article is organized as follows. In Section~\ref{sec:quadratic}, we study the special case of optimizing a strongly convex quadratic under the overparameterized noise assumption. Moreover, we analyze the resulting convergence rate $m(\gamma)$ (which is also the asymptotic convergence rate in our main result, Theorem~\ref{theo1}), optimize it with respect to the step-size and compare it to the (S)GD rate for $C^{1,1}$-functions that satisfy the PL-inequality derived in~\cite{wojtowytsch2021stochastic}. The remaining sections are devoted to the proof of the main result. In Section~\ref{sec:descent}, we establish a technical descent lemma. In Section~\ref{sec:geometric}, we construct a suitable manifold chart and control certain Taylor approximation errors. Finally, in Section~\ref{sec:proof}, we complete the proof of Theorem~\ref{theo1}.

\section{Convergence rate $m(\gamma)$: Quadratic derivation and properties} \label{sec:quadratic}
\subsection*{Convergence rates for quadratic objectives}
To illustrate the key mechanisms behind our main result, Theorem~\ref{theo1}, we consider the case of optimizing a strongly convex quadratic
$$
    f(\theta) = \frac 12 \langle A\theta, \theta \rangle,
$$
where $A\in \R^{d\times d}$ is a symmetric matrix that satisfies $\spec(A) \subset [\mu, L]$ for $0<\mu\le L$.
In particular, $A$ is positive definite and $\theta^\ast=0$ is the unique minimum. Since  $\nabla f(\theta) = A \theta$, the SGD iteration becomes
\begin{align} \label{eq:2304782711}
 \theta_k =  \theta_{k-1} - \gamma (A  \theta_{k-1} + D_k)=(\Id - \gamma A)\theta_{k-1}-\gamma D_k,
\end{align}
where $\Id$ denotes the identity matrix on $\R^d$.
We assume that $\E[\|\theta_0\|^2]<\infty$ and that there exists $\sigma>0$ such that for all $k \in \N$
\begin{align} \label{eq:00212}
\mathbb{E}[D_k \mid \cF_{k-1}]=0 \quad \text{ and } \quad \E[\|D_k\|^2 \mid \mathcal F_{k-1}] \le \sigma (f(\theta_{k-1})-f(\theta^\ast))= \frac{\sigma}{2} \langle A \theta_{k-1}, \theta_{k-1} \rangle.
\end{align}
In particular, $\theta^\ast=0$ is $(\mu,L,\sigma)$-regular in the sense of our Definition \ref{ass:PL}.
Using \eqref{eq:2304782711} together with \eqref{eq:00212}, yields for all $k \in \N$
\begin{align}
\begin{split}\label{eq:34341089247}
    \E[ \|\theta_k\|^2 \mid \mathcal F_{k-1}] &= \E[  \|(\Id - \gamma A) \theta_{k-1}\|^2 -2 \gamma \langle (\Id - \gamma A) \theta_{k-1}, D_k \rangle + \gamma^2 \|D_k\|^2  \mid \mathcal F_{k-1}] \\
    & \le \|(\Id - \gamma A) \theta_{k-1}\|^2 + \gamma^2  \sigma (f(\theta_{k-1})-f(\theta^\ast))\\
    &= \langle \theta_{k-1},\Big((\Id - \gamma A)^2+\frac{\gamma^2\sigma }{2}A\Big)\theta_{k-1}\rangle \\
    & \le \biggl \|(\Id - \gamma A)^2+\frac{\gamma^2\sigma }{2}A\biggr \| \, \|\theta_{k-1}\|^2 \\
    & = \left(\sup_{\lambda \in \spec(A)}(1 - \gamma \lambda)^2+\frac{\gamma^2\sigma \lambda}{2}\right)\|\theta_{k-1}\|^2\\
    & \le \max\left((1 - \gamma \mu)^2+\frac{\gamma^2\sigma \mu}{2}, (1 - \gamma L)^2+\frac{\gamma^2\sigma L}{2}\right)\|\theta_{k-1}\|^2.
\end{split}
\end{align}
Iterating this inequality gives
\begin{equation}\label{eq:381403048}
    \E [ \|\theta_k\|^2 \mid \cF_0 ]
\le 
m(\gamma)^k
\|\theta_0\|^2
\end{equation}
with
\begin{equation}\label{eq:def_m_quadratic}
    m(\gamma):=\max\left((1 - \gamma \mu)^2+\frac{\gamma^2\sigma \mu}{2}, (1 - \gamma L)^2+\frac{\gamma^2\sigma L}{2}\right).
\end{equation}
Therefore, provided $m(\gamma)<1$, we obtain linear convergence of $\|\theta_k\|^2$ in expectation to $\theta^\ast=0$. Now, fix an $\eps>0$ such that $m(\gamma)+\eps<1$.
Then, by Markov's inequality, for every $\delta>0$,
$$
\mathbb P\Big(\|\theta_k\|^2\ge \delta(m(\gamma)+\eps)^k\Big)
\le
\frac{\E[\|\theta_k\|^2]}{\delta(m(\gamma)+\eps)^k}
\le
\frac{\E[\|\theta_0\|^2]}{\delta}
\left(
\frac{m(\gamma)}
{m(\gamma)+\eps}
\right)^k.
$$
It follows for every $\delta>0$ that
$$
\sum_{k \in \N_0}\mathbb P\Big(\|\theta_k\|^2\ge \delta (m(\gamma)+\eps)^k\Big)<\infty,
$$
so that, by the Borel-Cantelli lemma, 
\begin{equation}\label{eq:rate_quadratic}
    \limsup_{k \to \infty} (m(\gamma)+\eps)^{-k} \|\theta_k\|^2=0, \quad \text{ almost surely}.
\end{equation}
Observe that $m(\gamma)$ given by \eqref{eq:def_m_quadratic} is precisely the asymptotic convergence rate obtained in Theorem~\ref{theo1}. The theorem thus shows that the asymptotic convergence rate in \eqref{eq:rate_quadratic} not only holds for the toy model of purely quadratic objectives but extends locally under suitable assumptions to general $C^2$-objectives satisfying the PL-inequality.

\subsection*{Optimality of $m(\gamma)$} Note that without imposing further assumptions, the rate $m(\gamma)$ given by \eqref{eq:def_m_quadratic} is optimal and cannot be improved. Indeed, consider \eqref{eq:2304782711} with $A$ having minimal and maximal eigenvalue $\mu$ and $L$, respectively. Assume that $D_k=\eta_k A^{1/2}\theta_{k-1}$, where $(\eta_k)_{k \in \N}$ is an i.i.d.\ sequence of univariate, square integrable random variables satisfying $\E[\eta_1]=0$ and $\E[\eta_1^2]=\frac {\sigma}{2}$. Then \eqref{eq:00212} is satisfied for the choice $(\mathcal F_k)_{k \in \N_0} = (\sigma(\eta_1, \dots, \eta_k))_{k \in \N_0}$ (the second inequality being an equality). Moreover, let $\lambda_{\max} \in \{\mu, L\}$ denote a maximizer in $
\max_{\lambda \in \spec(A)} \left\{ (1 - \gamma \lambda)^2 + \frac{\gamma^2 \sigma \lambda}{2} \right\}.
$
Let $\theta_0$ be in the eigenspace of $A$ associated with the eigenvalue $\lambda_{\max}$. Then $\theta_0$ is an eigenvector of $(\Id - \gamma A)^2+\frac{\gamma^2\sigma }{2}A$ with eigenvalue $\lambda_{\max}$. By induction, it follows that all $\theta_k$, $k\in \N$, are eigenvectors of $(\Id - \gamma A)^2+\frac{\gamma^2\sigma }{2}A$ with eigenvalue $\lambda_{\max}$. It thus follows that all inequalities in \eqref{eq:34341089247} are equalities in this set-up, and hence we also have equality in \eqref{eq:381403048}.

\subsection*{Step-size conditions ensuring convergence ($m(\gamma)<1$)}
It follows directly from the definition of $m$ that $m(\gamma)<1$ if and only if
$$
    \max \Big( -2\mu\gamma + \Big(\mu^2 + \frac{\sigma \mu}{2}\Big)\gamma^2, -2L\gamma + \Big(L^2 + \frac{\sigma L}{2}\Big)\gamma^2 \Big) < 0.
$$
This simplifies to 
$$
    \max \Big( \mu \gamma\left(\left(\mu+\frac{\sigma}{2}\right)\gamma-2\right), L\gamma\Big(\Big(L+\frac{\sigma}{2}\Big)\gamma-2\Big) \Big) < 0,
$$
which, using $0<\mu\le L$, is equivalent to 
$$
    0<\gamma<\frac{2}{L+\frac{\sigma}{2}}.
$$
Note that this is precisely the step-size condition we impose in Theorem \ref{theo1}; see \eqref{eq:step_size_cond_main_thm}.

\subsection*{Optimal step-sizes}
We minimize $m(\gamma)$ with respect to $\gamma$. First, note that
$$
(1 - \gamma L)^2+\frac{\gamma^2\sigma L}{2}-(1 - \gamma \mu)^2-\frac{\gamma^2\sigma \mu}{2}
=
\gamma(L-\mu)\Big(\gamma\Big(\frac{\sigma}{2}+L+\mu\Big)-2\Big)
$$
and hence
\begin{align} \label{eq:mgamma1}
m(\gamma)
&= \begin{cases}
(1-\gamma\mu)^2+\frac{\gamma^2\sigma \mu}{2}, & \text{if } 0\le \gamma\le \frac{2}{L+\mu+\frac{\sigma}{2}},\\
(\gamma L-1)^2+ \frac{\gamma^2\sigma L}{2}, & \text{if }\gamma >  \frac{2}{L+\mu+\frac{\sigma}{2}}.
\end{cases}
\end{align}
The minimizer of the first branch is
$\frac{1}{\mu+\frac{\sigma }{2}}$.
Moreover, $m$ is increasing on the second branch. 
Hence, the global minimizer of $m$ is
\begin{equation}\label{eq:opt_step_size_linear_case}
    \gamma^\ast=
\min\left(
\frac{1}{\mu+\frac{\sigma }{2}},
\;
\frac{2}{L+\mu+\frac{\sigma}{2}}
\right)
=
\begin{cases}
\frac{1}{\mu+\frac{\sigma }{2}},
& \text{if } \sigma \ge 2(L-\mu),
\\
\frac{2}{L+\mu+\frac{\sigma}{2}},
& \text{if } \sigma < 2(L-\mu).
\end{cases}
\end{equation}
This leads to
\begin{equation}\label{eq:our_opt_rate}
    m^\ast=m(\gamma^\ast) =     \begin{cases}
\frac{\sigma}{2\left(\mu+\frac{\sigma}{2}\right)},
& \text{if } \sigma \ge 2(L-\mu),\\
\frac{(L-\mu)^2 +\sigma(L+\mu)+ \frac{\sigma^2}{4}}{\left(L+\mu+\frac{\sigma}{2}\right)^2},
& \text{if } \sigma < 2(L-\mu).
\end{cases}
\end{equation}

\subsection*{Comparison to known PL-based rates}
We compare the asymptotic rate $m(\gamma)$ established in Theorem~\ref{theo1} with the non-asymptotic rate $\varphi(\gamma):=1-2\mu\gamma + \gamma^2 \frac{L(2\mu+\sigma)}{2}$ derived in~\cite{wojtowytsch2021stochastic} for $C^{1}$-objectives with Lipschitz-continuous gradient satisfying the PL-inequality. The latter result
provides a finite-time guarantee under weaker smoothness assumptions, while our geometric approach for $C^2$-functions yields a sharper asymptotic rate, as shown next.

First, note that $\varphi(\gamma)<1$ requires $\gamma<\frac{2}{L+ \frac{ \sigma L}{2 \mu} }$. For $\gamma \in (0,\frac{2}{L+\mu+\frac{\sigma}{2}}]$ it holds that
$$
\varphi(\gamma)-m(\gamma)=\gamma^2\Big(\mu+\frac{\sigma}{2}\Big)(L-\mu) \ge 0
$$
and for $\gamma\in(\frac{2}{L+\mu+\frac{\sigma}{2}},\frac{2}{L})$
we have
$$
\varphi(\gamma)-m(\gamma)=\gamma (2-\gamma L)(L-\mu) \ge 0
$$
with strict inequalities if $L>\mu$.
In particular, in that case we have $m(\gamma)<\varphi(\gamma)$ for all $\gamma\in (0,\frac{2}{L})$, so that for a given step-size, the asymptotic rate derived in our main result outperforms the non-asymptotic rate derived in \cite{wojtowytsch2021stochastic}.

Next, we compare the optimal rates. To this end, first note that $\varphi$ is minimized at
$
\frac{2\mu}{L(2\mu+\sigma)}
$
which gives a rate of convergence of
\begin{equation}\label{eq:opt_rate_PL}
  \varphi^\ast= 1-\frac{2\mu^2}{L(2\mu+\sigma)}. 
\end{equation}
We now compare the ratio $\frac{1-m^\ast}{1-\varphi^\ast}$ (see \eqref{eq:our_opt_rate} for the definition of $m^\ast$), which quantifies the relative contraction per iteration. To this end, we introduce the condition number $\kappa=\frac{L}{\mu}\ge 1$. Consider first the low condition regime $\kappa\le 1+\frac{\sigma}{2\mu}$ (i.e. $\sigma\ge 2(L-\mu)$).
In this case, it follows from \eqref{eq:our_opt_rate} and \eqref{eq:opt_rate_PL} that 
$$
\frac{1-m^\ast}{1-\varphi^\ast}=\kappa \in \left(1, 1+\frac{\sigma}{2\mu}\right].
$$
Consider next the high condition regime $\kappa\ge 1+\frac{\sigma}{2\mu}$ (i.e. $\sigma\le  2\mu (\kappa-1)$). In this case, we have 
$$
\frac{1-m^\ast}{1-\varphi^\ast}=\frac{4\kappa^2(1+\frac{\sigma}{2\mu})}{(\kappa+1+\frac{\sigma}{2\mu})^2}\in \left[1+\frac{\sigma}{2\mu} , 4\big(1+\frac{\sigma}{2\mu}\big)\right).
$$
In particular, we obtain in the ill-conditioned limit
$$
\lim_{\kappa \to \infty}\frac{1-m^\ast}{1-\varphi^\ast}=4\left(1+\frac{\sigma}{2\mu}\right).
$$
In conclusion, we thus establish an asymptotic speed-up in the convergence of SGD on $C^2$-functions under a local PL-inequality that is strictly increasing in the condition number $\kappa$ and approaches the maximal level of $4\left(1+\frac{\sigma}{2\mu}\right)$ in the ill-conditioned limit.

\section{A Descent Lemma in Transformed Coordinates} \label{sec:descent}
In this section, we study the SGD scheme
\begin{align} \label{eq:SGD2}
    \theta_k =\theta_{k-1} - \gamma(\nabla f(\theta_{k-1})+D_k), \qquad k\in \N,
\end{align}
in local coordinates induced by a chart $\Phi: U \to \Phi(U) \subset \R^d$, where $U \subset \R^d$ is an open set. We restrict our analysis to the event that $(\theta_k)_{k \in \N_0}$ stays in a small subset of the domain $U$ from some iteration $N\in \N$ onward. A first-order Taylor expansion yields
\begin{align*}
    \Phi(\theta_k) &= \Phi(\theta_{k-1}- \gamma(\nabla f(\theta_{k-1})+D_k)) \approx \Phi(\theta_{k-1}) - \gamma D\Phi(\theta_{k-1}) (\nabla f(\theta_{k-1})+D_k).
\end{align*}
Thus, in the coordinates induced by $\Phi$, the SGD dynamics are approximately governed by $D\Phi \cdot \nabla f$, which is the push-forward of the vector field $\nabla f$ under $\Phi$.

We derive a descent lemma for the distance between $\Phi(\theta_k)$ and the linear subspace $\R^{d_{\cT}} \times \{0\}^{d_{\cN}}$ which contains the image of the local minima of $f$ under $\Phi$. Near critical points, certain error terms vanish, and iterating the statement of the following lemma will yield a convergence rate that asymptotically matches the one obtained in the quadratic setting in Section~\ref{sec:quadratic}. We will substantiate this claim in the following sections.

Let $\proj_{\cT}: \R^d \to \R^{d_{\cT}} \times \{0\}^{d_{\cN}}$ denote the orthogonal projection onto $\R^{d_{\cT}} \times \{0\}^{d_{\cN}}$, and let $\proj_{\cN}: \R^d \to \{0\}^{d_{\cT}} \times \R^{d_{\cN}}$ denote the orthogonal projection onto $\{0\}^{d_{\cT}} \times \R^{d_{\cN}}$.

\begin{lemma} \label{prop:rate}
Let $U \subset \R^d$ be an open set, let $\Phi:U \to \R^d$ be a $C^1$-diffeomorphism, let $f:\R^d \to \R$ be continuously differentiable, and let $(\theta_k)_{k \in \N}$ be the SGD scheme defined in \eqref{eq:SGD2} .
    Assume that there exist constants $C_{\Phi,1}, C_{\Phi,2}, C_{f,1}, C_{f,2}, L \ge 0$ and a family $(A(\theta))_{ \theta \in U}$ of matrices in $\R^{d \times d}$ such that for all $\theta, \theta' \in U$
    \begin{itemize}
        \item[(i)]
        $(\nabla f)^{-1}(\{0\})\cap U =  \Phi^{-1}((\R^{d_{\cT}} \times \{0\}^{d_{\cN}} )\cap \Phi(U))$,
        \item [(ii)] $\|D \Phi(\theta)-D\Phi(\theta')\| \le C_{\Phi,1}$, 
        \item[(iii)] $\|D\Phi(\theta)\|\le C_{\Phi,2}$,
        \item[(iv)] $\|D\Phi(\theta)\nabla f(\theta)-A(\theta) \proj_{\cN}(\Phi(\theta)) \| \le C_{f,1} \|\proj_{\cN}(\Phi(\theta))\|$ and $\|A(\theta)\| \le L$, and
        \item[(v)] $\|\nabla f(\theta)\| \le C_{f,2}\|\proj_{\cN}(\Phi(\theta))\|$.
    \end{itemize}
    Moreover, let $N \in \N_0$, $C_{D,1}, \delta >0$, and let $\tilde U\subset U$ be an open set such that $\bigcup_{\theta \in \tilde U} \overline{B_{2\delta}(\theta)} \subset U$ and $\|\nabla f\|_{\infty,\tilde U}\le \delta/\gamma$.
    Let $\IB_{N}:= \{\theta_N \in \tilde U\}$, and, for $k > N$, let 
    $$
        \IB_k := \bigcap_{i=N, \dots, k-1} \{\theta_i \in \tilde U\} \cap \bigcap_{i=N,\dots, k}\{\|D_i\| \le \delta/\gamma \},
    $$
    and assume that for all $k > N$ one has almost surely on $\{\theta_{k-1} \in \tilde U\}$ that $\E[D_k \mid \mathcal F_{k-1}]=0$ and 
    \begin{enumerate}
        \item[(vi)] $
        \E[\|D_k\|^2 \mid \mathcal F_{k-1}] \le \frac {\sigma}{2} \langle A(\theta_{k-1}) \proj_{\cN}(\Phi(\theta_{k-1})), \proj_{\cN}(\Phi(\theta_{k-1}))  \rangle 
        + C_{D,1} \|\proj_{\cN}(\Phi(\theta_{k-1}))\|^2.
        $
    \end{enumerate}
    \medskip
    
    Then, for all $k > N$
    \begin{align*}
    \E[&\1_{\IB_k}\|\proj_{\cN}(\Phi(\theta_k))\|^2 \mid \cF_{k-1}] \\
    \le & \1_{\IB_{k-1}} \Big( \|(\Id-\gamma A(\theta_{k-1})) \proj_{\cN}(\Phi(\theta_{k-1}))\|^2 + C_{\Phi,2}^2 \frac{\gamma^2 \sigma}{2} \langle A(\theta_{k-1}) \proj_{\cN}(\Phi(\theta_{k-1})), \proj_{\cN}(\Phi(\theta_{k-1})) \rangle\Big) \\
        &+ \1_{\IB_{k-1}}\|\proj_{\cN}(\Phi(\theta_{k-1}))\|^2 \Bigg( C_{\Phi,2}^2 C_{D,1} \frac{\gamma^2 \sigma}{2} + 2 C_{\Phi,1} \gamma \Big( C_{f,2} + C_{D,2} +  \frac{\gamma}{2} C_{\Phi,2} ( C_{f,2}+ C_{D,2} )^2 \Big)  \\
        & + 2 \gamma C_{f,1} \Big( 1+\frac{\gamma}{2} C_{\Phi,2} (C_{f,2}+C_{D,2})  \Big) +\gamma^2  L \Big(\frac{\gamma}{\delta}C_{\Phi,2} R C_{D,2}^2 + C_{f,1}   + C_{\Phi,1} C_{f,2}+ C_{\Phi,1} C_{D,2}\Big)   \\
        & + \gamma^2 \Big( C_{\Phi,1} C_{\Phi,2} (C_{f,2}C_{D,2}+C_{D,2}^2) + \frac{C_{\Phi,2}}{\delta}  C_{D,2}^2 R ( 2 + \gamma C_{\Phi,2}C_{f,2} ) \Big) \Bigg),
    \end{align*}
    with $R:= \sup_{\theta \in \tilde U} \|\proj_{\cN}(\Phi(\theta))\|$ and $C_{D,2}^2:= L \frac{\sigma}{2} + C_{D,1}$.
\end{lemma}

\begin{proof} 
    We introduce the auxiliary events $\IA_{N}:= \{\theta_N \in \tilde U\}$ 
    $$
        \IA_k := \{\theta_N \in \tilde U\} \cap \Big( \bigcap_{i=N+1, \dots, k} \big(\{ \theta_i \in \tilde U \} \cap \{\| D_{i}\| \le \delta/\gamma\}\big)\Big) \qquad (k > N),
    $$
    so that $\IB_{k}:= \IA_{k-1} \cap \{\|D_{k}\|\le \delta/\gamma\}$ for all $k>N$.
    
    Let $k >N$. For $s \in [0, \gamma]$ define $\theta_{k-1}(s)= \theta_{k-1}-s(\nabla f(\theta_{k-1})+D_k)$ and note that, on the event $\IB_k$, one has $\theta_{k-1} \in \tilde U$ and, thus,
    $$
        d(\theta_{k-1}(s), \tilde U) \le \|\theta_{k-1}(s)-\theta_{k-1}\|\le \gamma (\|\nabla f(\theta_{k-1})\|+ \|D_k\|) \le 2 \delta 
    $$
    for all $s \in [0,\gamma]$. We conclude that $\theta_{k-1}(s) \in U$ for all $s \in [0,\gamma]$ and, in particular, $\theta_k \in U$ on $\IB_k$.

    For $\theta \in U$, we denote $\Phi_{\cN}(\theta):=\proj_{\cN}(\Phi(\theta))$ and consider the function $\theta \mapsto \|\Phi_{\cN}(\theta)\|^2$, which is the squared distance of $\Phi(\theta)$ to the subspace $\R^{d_{\cT}} \times \{0\}^{d_{\cN}}$ and has the derivative
    $$
        \xi(\theta) = 2 D\Phi(\theta )^\top \Phi_{\cN}(\theta).
    $$
    Thus, for all $k \in \N$, we can decompose $\|\Phi_{\cN}(\theta_k)\|^2$ on $\IB_k$ as follows
    \begin{align}\label{eq:3242355345}
    \begin{split}
        &\|\Phi_{\cN}(\theta_k)\|^2 = \|\Phi_{\cN}(\theta_{k-1})\|^2 - \int_0^\gamma \langle \xi(\theta_{k-1}(s)), \nabla f(\theta_{k-1}) + D_k \rangle \, ds \\
        &= \|\Phi_{\cN}(\theta_{k-1})\|^2- 2\int_0^\gamma \langle  (D\Phi(\theta_{k-1}(s) )-D\Phi(\theta_{k-1}))^\top \Phi_{\cN}(\theta_{k-1}(s)), \nabla f(\theta_{k-1}) + D_k \rangle \, ds \\
        &\quad - 2\int_0^\gamma \langle   \Phi_{\cN}(\theta_{k-1}(s)), D\Phi(\theta_{k-1})\nabla f(\theta_{k-1})\rangle \, ds
        - 2\int_0^\gamma \langle  \Phi_{\cN}(\theta_{k-1}(s)), D\Phi(\theta_{k-1}) D_k \rangle \, ds.
        \end{split}
    \end{align}
 In the following, we estimate the expectation of the three integrals in \eqref{eq:3242355345} on $\IB_k$ separately.
 
    First note that on $\IB_k$ we have
    \begin{align}\label{eq:239846871}
        \| (D\Phi(\theta_{k-1}(s))- D\Phi(\theta_{k-1}))^\top \Phi_{\cN}(\theta_{k-1}(s))\| &\le C_{\Phi,1} \| \Phi_{\cN}(\theta_{k-1}(s))\| 
    \end{align}
    and
    \begin{align} \begin{split} \label{eq:239846872}
        &\|\Phi_{\cN}(\theta_{k-1}(s))\| = \|\Phi_{\cN}(\theta_{k-1}-s(\nabla f(\theta_{k-1})+D_k))\| \\
        & = \Big\|\proj_{\cN}\Big(\Phi(\theta_{k-1})-\int_0^s D\Phi(\theta_{k-1}(u))(\nabla f(\theta_{k-1})+D_k) \, du \Big) \Big\| \\
        & \le \|\Phi_{\cN}(\theta_{k-1})\| + s C_{\Phi,2} (\|\nabla f(\theta_{k-1})\| +\|D_k\|).
        \end{split}
    \end{align}
    Combining \eqref{eq:239846871} and \eqref{eq:239846872} yields
    \begin{align*}
        &\E \Big[ \1_{\IB_k} \Big \| \int_0^\gamma \langle (D\Phi(\theta_{k-1}(s))- D\Phi(\theta_{k-1}))^\top \Phi_{\cN}(\theta_{k-1}(s)) , \nabla f(\theta_{k-1}) + D_k \rangle \, ds \Big\| \mid \cF_{k-1} \Big] \\
        \le & \E \Big[ \1_{\IB_k}   C_{\Phi,1} \int_0^\gamma \Big(\| \Phi_{\cN}(\theta_{k-1})\|+ s C_{\Phi,2} (\|\nabla f(\theta_{k-1})\| +  \|D_k\|) \Big)  \, \Big( \| \nabla f(\theta_{k-1})\| + \|D_k \|\Big) \, ds  \mid \cF_{k-1} \Big] \\
        \le & \E \Big[ \1_{\IB_k}   C_{\Phi,1} \gamma \Big(\| \Phi_{\cN}(\theta_{k-1})\|+ \frac{\gamma}{2} C_{\Phi,2} (C_{f,2}\|\Phi_{\cN}(\theta_{k-1})\| +  \|D_k\|) \Big)  \, \Big( C_{f,2}\|\Phi_{\cN}(\theta_{k-1})\| + \|D_k \|\Big)  \mid \cF_{k-1} \Big].
        \end{align*}
        Using $\IB_{k} \subset \IA_{k-1}$, and the fact that 
    \begin{align*}
        &\E[\1_{\IA_{k-1}} \|D_k\| \mid \mathcal F_{k-1}] \le \E[\1_{\IA_{k-1}} \|D_k\|^2 \mid \mathcal F_{k-1}]^{1/2} \\
        &\le \1_{\IA_{k-1}} \Big( \frac {\sigma}{2} \langle A(\theta_{k-1}) \Phi_{\cN}(\theta_{k-1}), \Phi_{\cN}(\theta_{k-1}) \rangle+C_{D,1} \|\Phi_{\cN}(\theta_{k-1})\|^2 \Bigr)^{1/2} \le \1_{\IA_{k-1}} C_{D,2} \|\Phi_{\cN}(\theta_{k-1})\|
    \end{align*}
    for $C_{D,2}^2:= L \frac{\sigma}{2}+C_{D,1}$ we thus obtain
        \begin{align}\label{eq:342313a}
        \begin{split}
        &\E \Big[ \1_{\IB_k} \Big \| \int_0^\gamma \langle (D\Phi(\theta_{k-1}(s))- D\Phi(\theta_{k-1}))^\top \Phi_{\cN}(\theta_{k-1}(s)) , \nabla f(\theta_{k-1}) + D_k \rangle \, ds \Big\| \mid \cF_{k-1} \Big]\\
        \le & \1_{\IA_{k-1}} \Big( C_{\Phi,1} \gamma \Big( C_{f,2} + C_{D,2} + \frac{\gamma}{2} C_{\Phi,2} ( C_{f,2}+ C_{D,2}  )^2 \Big)  \|\Phi_{\cN}(\theta_{k-1})\|^2\Big).
        \end{split}
    \end{align}

To estimate the second integral in \eqref{eq:3242355345},
   first note that for all $s \in [0, \gamma]$ we have on $\IB_k$
    \begin{align} \begin{split} \label{eq:9921}
        &\langle \Phi_{\cN}(\theta_{k-1}(s)), D\Phi(\theta_{k-1}) \nabla f(\theta_{k-1}) \rangle \\
        \ge & \langle \Phi_{\cN}(\theta_{k-1}(s)), A(\theta_{k-1}) \Phi_{\cN}(\theta_{k-1}) \rangle - C_{f,1} \|\Phi_{\cN}(\theta_{k-1})\| \, \|\Phi_{\cN}(\theta_{k-1}(s))\| \\
        \ge & \langle \Phi_{\cN}(\theta_{k-1}(s)), A(\theta_{k-1}) \Phi_{\cN}(\theta_{k-1}) \rangle \\
        &- C_{f,1} \|\Phi_{\cN}(\theta_{k-1})\| \, \big(\|\Phi_{\cN}(\theta_{k-1})\| + s C_{\Phi,2} \big(\|\nabla f(\theta_{k-1})\| +\|D_k\|\big) \big),
        \end{split}
    \end{align}
    where we used \eqref{eq:239846872} in the last inequality.
    Moreover, on $\IB_k$ we have
    \begin{equation*}
        \begin{split}
            &\langle \Phi_{\cN}(\theta_{k-1}(s))-\Phi_{\cN}(\theta_{k-1}) , A(\theta_{k-1}) \Phi_{\cN}(\theta_{k-1}) \rangle\\
            &=
            \langle \proj_{\cN}(\Phi(\theta_{k-1}(s))-\Phi(\theta_{k-1})+ D\Phi(\theta_{k-1})s(\nabla f(\theta_{k-1})+D_k)), A(\theta_{k-1}) \Phi_{\cN}(\theta_{k-1}) \rangle\\
            &\quad - s\langle \proj_{\cN}(D\Phi(\theta_{k-1})(\nabla f(\theta_{k-1})+D_k)), A(\theta_{k-1}) \Phi_{\cN}(\theta_{k-1}) \rangle\\
            &\ge -\|\Phi(\theta_{k-1}(s))-\Phi(\theta_{k-1})+ D\Phi(\theta_{k-1})s(\nabla f(\theta_{k-1})+D_k)\| \, \|A(\theta_{k-1}) \Phi_{\cN}(\theta_{k-1})\|\\
            &\quad - s\langle \proj_{\cN}(D\Phi(\theta_{k-1})\nabla f(\theta_{k-1})), A(\theta_{k-1}) \Phi_{\cN}(\theta_{k-1}) \rangle
            - s\langle \proj_{\cN}(D\Phi(\theta_{k-1})D_k), A(\theta_{k-1}) \Phi_{\cN}(\theta_{k-1}) \rangle\\
            &\ge -C_{\Phi,1} s \|\nabla f(\theta_{k-1})+D_k\| \, \|A(\theta_{k-1}) \Phi_{\cN}(\theta_{k-1})\|\\
            &\quad - s\langle \proj_{\cN}(D\Phi(\theta_{k-1})\nabla f(\theta_{k-1})-A(\theta_{k-1}) \Phi_{\cN}(\theta_{k-1})), A(\theta_{k-1}) \Phi_{\cN}(\theta_{k-1}) \rangle-s\|A(\theta_{k-1}) \Phi_{\cN}(\theta_{k-1})\|^2\\
            &\quad - s\langle \proj_{\cN}(D\Phi(\theta_{k-1})D_k), A(\theta_{k-1}) \Phi_{\cN}(\theta_{k-1}) \rangle\\
            &\ge -C_{\Phi,1} s \|\nabla f(\theta_{k-1})+D_k\| \,  \|A(\theta_{k-1}) \Phi_{\cN}(\theta_{k-1})\| - sC_{f,1}\|\Phi_{\cN}(\theta_{k-1})\| \, \|A(\theta_{k-1}) \Phi_{\cN}(\theta_{k-1})\|\\
            &\quad -s\|A(\theta_{k-1}) \Phi_{\cN}(\theta_{k-1})\|^2- s\langle \proj_{\cN}(D\Phi(\theta_{k-1})D_k), A(\theta_{k-1}) \Phi_{\cN}(\theta_{k-1}) \rangle.
        \end{split}
    \end{equation*}
Taking conditional expectation yields
    \begin{equation}\label{eq:auxiliary1}
    \begin{split}
        &\E[ \1_{\IB_{k}} \langle \Phi_{\cN}(\theta_{k-1}(s))-\Phi_{\cN}(\theta_{k-1}), A(\theta_{k-1}) \Phi_{\cN}(\theta_{k-1}) \rangle \mid \mathcal F_{k-1}] \\
        &\ge -s(C_{\Phi,1} \E[ \1_{\IB_{k}} \|\nabla f(\theta_{k-1})+D_k\|  \mid \mathcal F_{k-1}]+ \1_{\IA_{k-1}}C_{f,1}\|\Phi_{\cN}(\theta_{k-1})\|) \, \|A(\theta_{k-1}) \Phi_{\cN}(\theta_{k-1})\|\\
       & \quad -s\E[ \1_{\IB_{k}} \|A(\theta_{k-1}) \Phi_{\cN}(\theta_{k-1})\|^2+\1_{\IB_{k}} \langle \proj_{\cN}(D\Phi(\theta_{k-1})D_k), A(\theta_{k-1}) \Phi_{\cN}(\theta_{k-1}) \rangle\mid \mathcal F_{k-1}] \\
        &\ge -s\1_{\IA_{k-1}}(C_{\Phi,1}(C_{f,2}+C_{D,2}) + C_{f,1})\|\Phi_{\cN}(\theta_{k-1})\| \, \|A(\theta_{k-1}) \Phi_{\cN}(\theta_{k-1})\|\\
       & \quad -s\E[ \1_{\IB_{k}} \|A(\theta_{k-1}) \Phi_{\cN}(\theta_{k-1})\|^2+\1_{\IB_{k}} \langle \proj_{\cN}(D\Phi(\theta_{k-1})D_k), A(\theta_{k-1}) \Phi_{\cN}(\theta_{k-1}) \rangle\mid \mathcal F_{k-1}]
        \end{split}
    \end{equation}
  Moreover, using $\E[\1_{\IA_{k-1}} D_k\mid \cF_{k-1}]=0$ and $\1_{\IB_{k}}=\1_{\IA_{k-1}}-\1_{\IA_{k-1} \cap \{\|D_k\| > \delta/\gamma\}}$, we obtain
    \begin{align*}
        &\E [ \1_{\IB_k} \langle \proj_{\cN}(D\Phi(\theta_{k-1}) D_k), A(\theta_{k-1}) \Phi_{\cN}(\theta_{k-1}) \rangle \mid \cF_{k-1}] \\
         =& - \E [ \1_{\IA_{k-1}\cap \{\|D_k\| > \delta/\gamma\}} \langle \proj_{\cN}(D\Phi(\theta_{k-1}) D_k), A(\theta_{k-1}) \Phi_{\cN}(\theta_{k-1}) \rangle \mid \cF_{k-1}] \\
         \le &\E [ \1_{\IA_{k-1}\cap \{\|D_k\| > \delta/\gamma\}} \|D\Phi(\theta_{k-1}) D_k\| \mid \cF_{k-1}]\,  \|A(\theta_{k-1}) \Phi_{\cN}(\theta_{k-1})\| \\
         \le &\frac{\gamma}{\delta} C_{\Phi,2} \E [ \1_{\IA_{k-1}} \| D_k\|^2 \mid \cF_{k-1}]\,  \|A(\theta_{k-1}) \Phi_{\cN}(\theta_{k-1})\| \\
         \le &\1_{\IA_{k-1}}  \frac{\gamma}{\delta} C_{\Phi,2} R C_{D,2}^2 \|\Phi_{\cN}(\theta_{k-1})\| \,  \|A (\theta_{k-1}) \Phi_{\cN}(\theta_{k-1})\|.
    \end{align*}
    Combining this with \eqref{eq:auxiliary1} yields for all $s\in [0,\gamma]$
    \begin{align}\begin{split} \label{eq:99726392}
        &\E[ \1_{\IB_{k}} \langle \Phi_{\cN}(\theta_{k-1}(s))-\Phi_{\cN}(\theta_{k-1}), A(\theta_{k-1}) \Phi_{\cN}(\theta_{k-1}) \rangle \mid \mathcal F_{k-1}] \\
        \ge &- s\Big( \E \big[ \1_{\IB_{k}}\|A(\theta_{k-1})\Phi_{\cN}(\theta_{k-1})\|^2 \mid \cF_{k-1} \big]  \\
        &   \quad + \1_{\IA_{k-1}}\|A(\theta_{k-1}) \Phi_{\cN}(\theta_{k-1})\| \, \|\Phi_{\cN}(\theta_{k-1})\| \,  \Big(\frac{\gamma}{\delta}C_{\Phi,2} R C_{D,2}^2 + C_{f,1}   + C_{\Phi,1} C_{f,2}+ C_{\Phi,1} C_{D,2}\Big) \Big).
        \end{split}
    \end{align}

    Now, using \eqref{eq:9921}, \eqref{eq:99726392} and the fact that $\1_{\IA_{k-1}}\|A(\theta_{k-1})\|\le \1_{\IA_{k-1}} L$ gives
    \begin{align}\label{eq:342313b}
    \begin{split}
    &2\E\Big[\1_{\IB_{k}} \int_0^\gamma\langle  \Phi_{\cN}(\theta_{k-1}(s)), D\Phi(\theta_{k-1}) \nabla f(\theta_{k-1}) \rangle\, ds \Big) \mid \mathcal F_{k-1}\Big] \\
        & \ge \E\Big[\1_{\IB_{k}} \int_0^\gamma 2 \langle \Phi_{\cN}(\theta_{k-1}(s)), A(\theta_{k-1}) \Phi_{\cN}(\theta_{k-1}) \rangle \big) \, ds  \mid \mathcal F_{k-1}\Big] \\
        &  \quad - \1_{\IA_{k-1}} 2\gamma C_{f,1}\|\Phi_{\cN}(\theta_{k-1})\| \, \Big(\|\Phi_{\cN}(\theta_{k-1})\| + \frac{\gamma}{2} C_{\Phi,2} \big(\|\nabla f(\theta_{k-1})\| +\E[ \|D_k\| \mid \cF_{k-1}] \big)\Big)\\
        &\ge \E \big[\1_{\IB_{k}} 2\gamma \langle \Phi_{\cN}(\theta_{k-1}), A(\theta_{k-1}) \Phi_{\cN}(\theta_{k-1})  \rangle  - \gamma^2  \1_{\IB_{k}}\|A (\theta_{k-1})\Phi_{\cN}(\theta_{k-1}) \|^2 \mid \cF_{k-1} \big] \\
        &  \quad - \1_{\IA_{k-1}} \gamma^2  L \Big(\frac{\gamma}{\delta}C_{\Phi,2} R C_{D,2}^2 + C_{f,1}   + C_{\Phi,1} C_{f,2}+ C_{\Phi,1} C_{D,2}\Big)  \|\Phi_{\cN}(\theta_{k-1})\|^2 \\
        & \quad  - \1_{\IA_{k-1}} 2 \gamma C_{f,1} \Big( 1+\frac{\gamma}{2}  C_{\Phi,2} (C_{f,2}+C_{D,2})  \Big) \|\Phi_{\cN}(\theta_{k-1})\|^2.
        \end{split}
    \end{align}

    To estimate the third integral in \eqref{eq:3242355345} we note that, similarly to \eqref{eq:239846872}, on the event $\IB_k$ one has
    \begin{equation}\label{eq:auxiliary2}
    \begin{split}
         &  \langle \Phi_{\cN}(\theta_{k-1}(s)), D\Phi(\theta_{k-1})D_k \rangle   \\
         & = \langle \proj_{\cN}(\Phi(\theta_{k-1}) - \int_0^s D\Phi(\theta_{k-1}(u)) (\nabla f(\theta_{k-1})+D_k)  \, du), D\Phi(\theta_{k-1})D_k \rangle \\
         &\ge  \langle \proj_{\cN}(\Phi(\theta_{k-1}) -D\Phi(\theta_{k-1})s(\nabla f(\theta_{k-1})+D_k)), D\Phi(\theta_{k-1})D_k \rangle \\
         & \quad   -  C_{\Phi,1} C_{\Phi,2} s \|\nabla f(\theta_{k-1})+D_k\| \, \|D_k\|.
    \end{split}
    \end{equation}
    Now, by the fact that $\1_{\IB_{k}}=\1_{\IA_{k-1}}-\1_{\IA_{k-1} \cap \{\|D_k\| > \delta/\gamma\}}$ and using $\E[\1_{\IA_{k-1}} D_k\mid \cF_{k-1}]=0$ we obtain
    \begin{align*}
        &\E[ \1_{\IB_{k}}\langle \proj_{\cN}(\Phi(\theta_{k-1}) -D\Phi(\theta_{k-1})s\nabla f(\theta_{k-1})), D\Phi(\theta_{k-1}))D_k \rangle \mid \cF_{k-1}] \\
        & = -\E[ \1_{\IA_{k-1} \cap \{\|D_k\| > \delta/\gamma\}}\langle \proj_{\cN}( \Phi(\theta_{k-1}) -D\Phi(\theta_{k-1})s\nabla f(\theta_{k-1})), D\Phi(\theta_{k-1}))D_k \rangle \mid \cF_{k-1}] \\
        & \ge -\E[ \1_{\IA_{k-1} \cap \{\|D_k\| > \delta/\gamma\}} \| D\Phi(\theta_{k-1}))D_k \| \mid \cF_{k-1}] \, ( 1 + s C_{\Phi,2}C_{f,2} ) \|\Phi_{\cN}(\theta_{k-1})\| \\
        & \ge - \frac{\gamma}{\delta} C_{\Phi,2} \E[ \1_{\IA_{k-1}}  \|D_k \|^2 \mid \cF_{k-1}] \, ( 1 + s C_{\Phi,2}C_{f,2} ) \|\Phi_{\cN}(\theta_{k-1})\| \\
        & \ge - \1_{\IA_{k-1}} \frac{\gamma}{\delta} C_{\Phi,2} C_{D,2}^2 R ( 1 + s C_{\Phi,2}C_{f,2} ) \|\Phi_{\cN}(\theta_{k-1})\|^2,
    \end{align*}
        which together with \eqref{eq:auxiliary2} yields
    \begin{align}\label{eq:342313c}
    \begin{split}
        &\E\Big[\1_{\IB_k} \int_0^\gamma \langle \Phi_{\cN}(\theta_{k-1}(s)), D\Phi(\theta_{k-1})D_k \rangle \, ds \mid \cF_{k-1} \Big] \\
        &\ge  -  \frac{\gamma^2}{2} \E[\1_{\IA_{k-1}} \|D\Phi_{\cN}(\theta_{k-1}) D_k\|^2 \mid \mathcal F_{k-1}] \\
        &- \1_{\IA_{k-1}} \Big( \frac{\gamma^2}{2} C_{\Phi,1}C_{\Phi,2} (C_{f,2}C_{D,2}+C_{D,2}^2) + \frac{\gamma^2}{\delta} C_{\Phi,2} C_{D,2}^2 R \Big( 1 + \frac{\gamma}{2} C_{\Phi,2}C_{f,2} \Big) \Big) \|\Phi_{\cN}(\theta_{k-1})\|^2\\
        &\ge -\1_{\IA_{k-1}} \Big( C_{\Phi,2}^2 \frac{\gamma^2 \sigma}{4} \langle A(\theta_{k-1}) \Phi_{\cN}(\theta_{k-1}), \Phi_{\cN}(\theta_{k-1})  \rangle + C_{\Phi,2}^2 \frac{\gamma^2 \sigma}{4} C_{D,1} \|\Phi_{\cN}(\theta_{k-1})\|^2 \Big) \\
        &- \1_{\IA_{k-1}} \Big( \frac{\gamma^2}{2} C_{\Phi,1} C_{\Phi,2}(C_{f,2}C_{D,2}+C_{D,2}^2) + \frac{\gamma^2}{\delta} C_{\Phi,2} C_{D,2}^2 R \Big( 1 + \frac{\gamma}{2} C_{\Phi,2}C_{f,2} \Big) \Big) \|\Phi_{\cN}(\theta_{k-1})\|^2.
        \end{split}
    \end{align}
    
    In conclusion, we obtain by combining \eqref{eq:3242355345} with \eqref{eq:342313a}, \eqref{eq:342313b} and \eqref{eq:342313c} that 
    \begin{align*}
        &\E[\1_{\IB_k}\|\Phi_{\cN}(\theta_k)\|^2 \mid \cF_{k-1}]\\
        &\le \E \Big[ \1_{\IB_k} \Big( \|\Phi_{\cN}(\theta_{k-1})\|^2 - 2\gamma \langle \Phi_{\cN}(\theta_{k-1}), A(\theta_{k-1}) \Phi_{\cN}(\theta_{k-1})  \rangle  + \gamma^2 \|A(\theta_{k-1}) \Phi_{\cN}(\theta_{k-1})\|^2 \Big) \mid \cF_{k-1} \Big]\\
        & +\1_{\IA_{k-1}} \Big( C_{\Phi,2}^2 \frac{\gamma^2 \sigma}{4} \langle A(\theta_{k-1}) \Phi_{\cN}(\theta_{k-1}), \Phi_{\cN}(\theta_{k-1})  \rangle + C_{\Phi,2}^2 \frac{\gamma^2 \sigma}{4} C_{D,1} \|\Phi_{\cN}(\theta_{k-1})\|^2 \Big) \\
        &+ \1_{\IA_{k-1}}\|\proj_{\cN}(\Phi(\theta_{k-1}))\|^2 \Bigg( 2 C_{\Phi,1} \gamma \Big( C_{f,2} + C_{D,2} +  \frac{\gamma}{2} C_{\Phi,2} ( C_{f,2}+ C_{D,2}  )^2 \Big)  \\
        & + 2 \gamma C_{f,1} \Big( 1+\frac{\gamma}{2} C_{\Phi,2} (C_{f,2}+C_{D,2})  \Big) +\gamma^2  L \Big(\frac{\gamma}{\delta}C_{\Phi,2} R C_{D,2}^2 + C_{f,1}   + C_{\Phi,1} C_{f,2}+ C_{\Phi,1} C_{D,2}\Big)   \\
        & + \gamma^2 \Big( C_{\Phi,1} C_{\Phi,2} (C_{f,2}C_{D,2}+C_{D,2}^2) + \frac{C_{\Phi,2}}{\delta}  C_{D,2}^2 R ( 2 + \gamma C_{\Phi,2}C_{f,2} ) \Big) \Bigg),
    \end{align*}
    which yields the result, since $\IB_k \subset \IA_{k-1}\subset \IB_{k-1}$ and
    \begin{align*}
        &\|\Phi_{\cN}(\theta_{k-1})\|^2 - 2\gamma \langle \Phi_{\cN}(\theta_{k-1}), A(\theta_{k-1}) \Phi_{\cN}(\theta_{k-1})  \rangle  + \gamma^2 \|A(\theta_{k-1}) \Phi_{\cN}(\theta_{k-1})\|^2 \\
        &= \|(\Id-\gamma A(\theta_{k-1})) \Phi_{\cN}(\theta_{k-1})\|^2.
    \end{align*}
\end{proof}

In the following sections, we prove that Lemma~\ref{prop:rate} yields convergence with high probability once SGD enters a sufficiently small neighborhood of a $(\mu,L,\sigma)$-regular point. The argument follows the general strategy used for deterministic schemes in \cite[Section 3]{rebjock2023fast}; see also \cite{attouch2013convergence,ochs2018local}. We combine this framework with high-probability estimates to obtain almost sure convergence rates through Lemma~\ref{prop:rate} and an application of the Borel-Cantelli lemma.

\section{Geometric Preliminaries} \label{sec:geometric}
In this section, we control the error terms that arise in the descent lemma, Lemma~\ref{prop:rate}. Using Proposition \ref{thm:rebjock} we construct a manifold chart $\Phi: U \to \Phi(U)$ around a $(\mu,L, \sigma)$-regular point $\theta^\ast$ such that $D\Phi(\theta^\ast)$ is an orthogonal matrix. Then, for a sufficiently small neighborhood $U$ of $\theta^\ast$ the distances between the points $\theta,\theta' \in U$ are approximately equal to the distance of the images $\Phi(\theta),\Phi(\theta')$. In addition, the eigenvalues of $A(\theta) := D\Phi(\theta) \Hess f(\theta)(D\Phi(\theta))^{-1}$ are close to the eigenvalues of $\Hess f(\theta)$ and $\dim(\ker(A(\theta^\ast)))=d_{\cN}$. All remaining estimates are derived from the regularity of $f$ and $\Phi$. 

\begin{lemma} \label{prop:geometry}
    Let $\theta^\ast \in \R^d$ be a $(\mu,L, \sigma)$-regular point. Then, there exists a $C\ge 0$ such that for every $\eps>0$ there exist a neighborhood $U \subset \R^d$ of $\theta^\ast$ and a $C^1$-diffeomorphism $\Phi: U \to \Phi(U)$ with $\Phi(\theta^\ast)=0$, $D\Phi(\theta^\ast)$ is an orthogonal matrix and for all $\theta, \theta' \in U$
    \begin{itemize}
        \item[(i)] $f(\theta) \ge f(\theta^\ast)$ and $\nabla f^{-1}(\{0\})\cap U =  \Phi^{-1}((\R^{d_{\cT}} \times \{0\}^{d_{\cN}}) \cap \Phi(U))$,
        \item [(ii)] $\|D \Phi(\theta)-D\Phi(\theta')\| \le \eps$,
        \item[(iii)] $\|D\Phi(\theta)\|\le 1+\eps$,
        \item[(iv)] $\|D\Phi(\theta)\nabla f(\theta)-A(\theta) \proj_{\cN}(\Phi(\theta)) \| \le \eps \|\proj_{\cN}(\Phi(\theta))\|$, 
        \item[(v)] $\|\nabla f(\theta)\| \le C\|\proj_{\cN}(\Phi(\theta))\|$, and
        \item[(vi)] $f(\theta)-f(\theta^\ast) \le \frac 12 \langle A(\theta) \proj_{\cN}(\Phi(\theta)), \proj_{\cN}(\Phi(\theta))  \rangle 
        + \eps \|\proj_{\cN}(\Phi(\theta))\|^2,
        $
    \end{itemize}
    with $A(\theta) := D\Phi(\theta) (\Hess f)(\theta) (D\Phi(\theta))^{-1}$ for $\theta \in U$. Moreover, we have $A(\theta) \to A(\theta^\ast)$ as $\theta \to \theta^\ast$, where $A(\theta^\ast)$ is a symmetric matrix satisfying
    \begin{equation}\label{eq:43897589370}
        \ker(A(\theta^\ast)) = \R^{d_{\cT}} \times \{0\}^{d_{\cN}} \quad \text{ and } \quad \spec(A(\theta^\ast) \einschraenkung_{\{0\}^{d_{\cT}} \times \R^{d_{\cN}}}) \subset [\mu, L].
    \end{equation}
\end{lemma}

\begin{proof}
By Proposition~\ref{thm:rebjock}, there exists a $C^1$-diffeomorphism $\hat \Phi:U \to \hat \Phi(U)$, where $U\subset \R^d$ is a neighborhood of $\theta^\ast$, satisfying $f(\theta) \ge f(\theta^\ast)$ for all $\theta \in U$ and
\begin{align} \label{eq:mgfchart}
    \mathcal M := \nabla f^{-1}(\{0\})\cap U =  \hat \Phi^{-1}((\R^{d_{\cT}} \times \{0\}^{d_{\cN}}) \cap \hat \Phi(U)).
\end{align}
Let $O \in \R^{d \times d}$ be an orthogonal matrix such that $O \cdot \cT_{\theta^\ast}\mathcal M = \R^{d_{\cT}} \times \{0\}^{d_{\cN}}.$ Since the Jacobian matrix $D \hat \Phi(\theta^\ast)$ is invertible, we can define $\Phi(\theta):= O \cdot  (D \hat \Phi(\theta^\ast))^{-1} (\hat \Phi(\theta)-\hat \Phi(\theta^\ast))$, which is a $C^1$-diffeomorphism onto its image that satisfies all the above properties and, additionally,
$$
    \Phi(\theta^\ast)=0 \quad \text{ and } \quad 
    D \Phi (\theta^\ast)=O \cdot (D\hat \Phi(\theta^\ast))^{-1}D \hat \Phi(\theta^\ast)=O.
$$
In particular, \eqref{eq:mgfchart} still holds with $\hat \Phi$ replaced by $\Phi$ since 
$$
O \cdot  (D\hat \Phi(\theta^\ast))^{-1} (\R^{d_{\cT}} \times \{0\}^{d_{\cN}}) = O \cdot \cT_{\theta^\ast}\mathcal M = \R^{d_{\cT}} \times \{0\}^{d_{\cN}}.
$$

We show that, after possibly shrinking the neighborhood $U$, properties (ii)-(vi) hold.
First, note that after possibly shrinking $U$,
$\mathcal M := \nabla f^{-1}(\{0\})\cap U$ is connected. Since $\nabla f=0$ on
$\mathcal M$, one has $f(\theta)=f(\theta^\ast)$ for all $\theta \in \mathcal M$.
Moreover, we may assume that $D\Phi$ and $\Hess f$ are bounded on $U$ and $D (\Phi^{-1})$ is bounded on $\Phi(U)$. By Definition~\ref{ass:PL}, we finally may assume that $f$ is $C^2$ on a convex set containing $U$ and for every $\theta\in U$ and every $s\in[0,1]$
$$
\proj_{\cT}(\Phi(\theta))+s\proj_{\cN}(\Phi(\theta))\in \Phi(U).
$$

Using $\Phi \in C^1$, one has $D\Phi(\theta) \to D \Phi(\theta^\ast)=O$, as $\theta \to \theta^\ast$. Thus, properties (ii) and (iii) hold on a sufficiently small neighborhood of $\theta^\ast$.

Regarding property (iv), one has
\begin{align}\begin{split}\label{eq:3891543}
    &\|D\Phi(\theta)\nabla f(\theta)-A(\theta) \proj_{\cN}(\Phi(\theta)) \| \\
    &\le \|D \Phi(\theta)\| \, \|\nabla f(\theta)-\Hess f(\theta) (D\Phi(\theta))^{-1}\proj_{\cN}(\Phi(\theta))\|,
    \end{split}
\end{align}
where $\|D \Phi\|$ is bounded. Now, using that $\nabla f(\Phi^{-1}(\proj_{\cT}(\Phi(\theta)))=0$, we get by the fundamental theorem of calculus
\begin{align} \label{eq:hessbound}
    \nabla f(\theta) = \int_0^1 (\Hess f)((1-s)\Phi^{-1}(\proj_{\cT}(\Phi(\theta)))+s \theta)\,  (\theta-\Phi^{-1}(\proj_{\cT}(\Phi(\theta)))) \, ds.
\end{align}
Similarly, 
\begin{align} \label{eq:20372}
    \theta- \Phi^{-1}(\proj_{\cT}(\Phi(\theta))) = \int_0^1 D(\Phi^{-1})(\proj_{\cT}(\Phi(\theta))+s \proj_{\cN}(\Phi(\theta))) \proj_{\cN}(\Phi(\theta)) \, ds.
\end{align}
Substituting this identity into \eqref{eq:hessbound} yields
\begin{equation*}
    \begin{split}
      \nabla f(\theta)
&=
\biggl(\int_0^1
(\Hess f)((1-s)\Phi^{-1}(\proj_{\cT}(\Phi(\theta)))+s\theta)\,ds
\biggr)\\
&\qquad \cdot \biggl(
\int_0^1
D(\Phi^{-1})(\proj_{\cT}(\Phi(\theta))+t\proj_{\cN}(\Phi(\theta)))\,
\,dt
\biggr)
\proj_{\cN}(\Phi(\theta)).
    \end{split}
\end{equation*}
Using this together with \eqref{eq:3891543} and the continuity of $\Hess f$ and $D(\Phi^{-1})$ we conclude that 
$$
    \|D\Phi(\theta)\nabla f(\theta)-A(\theta) \proj_{\cN}(\Phi(\theta)) \| \le \eps \|\proj_{\cN}(\Phi(\theta))\|
$$
for all $\theta$ in a sufficiently small neighborhood $U$ of $\theta^\ast$. This proves (iv).

Regarding property (v), by the boundedness of $D(\Phi^{-1})$, there exists a constant $C_1 \ge 0$ such that for all $\theta \in U$
$$
    \|\theta-\Phi^{-1}(\proj_{\cT}(\Phi(\theta)))\| \le C_1 \|\proj_{\cN}(\Phi(\theta))\|,
$$ 
see \eqref{eq:20372}.
Since $\Hess f$ is bounded, say by $C_2\ge 0$, we use \eqref{eq:hessbound} to get
$$
    \|\nabla f (\theta)\| \le C_2 C_1 \|\proj_{\cN}(\Phi(\theta))\|,
$$
which proves (v).

Regarding (vi), note that for all $\theta \in U$ one has
\begin{align*}
    &f(\theta)-f(\theta^\ast) = f(\Phi^{-1}(\Phi(\theta)))-f(\Phi^{-1}(\proj_{\cT}(\Phi(\theta)))) \\
    & = \int_0^1 (\nabla f(\Phi^{-1}(\varphi_s)))^\top D\Phi^{-1}(\varphi_s) \proj_{\cN}(\Phi(\theta)) \, ds \\
    &= \int_0^1 \int_0^s  \big(\Hess f(\Phi^{-1}(\varphi_u)) D\Phi^{-1}(\varphi_u) \proj_{\cN}(\Phi(\theta)) \big)^\top D\Phi^{-1}(\varphi_s) \proj_{\cN}(\Phi(\theta)) \, du \, ds
\end{align*}
since $\nabla f(\Phi^{-1}(\proj_{\cT}(\Phi(\theta))))=0$. Using the continuity of $\Hess f$, $D(\Phi^{-1})$ and $A$ we conclude that
\begin{align*}
    &f(\theta)-f(\theta^\ast) \le \frac 12 \langle A(\theta^\ast) \proj_{\cN}(\Phi(\theta)), \proj_{\cN}(\Phi(\theta))  \rangle 
        + \frac{\eps}{2} \|\proj_{\cN}(\Phi(\theta))\|^2 \\
    & \le \frac 12 \langle A(\theta) \proj_{\cN}(\Phi(\theta)), \proj_{\cN}(\Phi(\theta))  \rangle 
        + \eps \|\proj_{\cN}(\Phi(\theta))\|^2
\end{align*}
for a sufficiently small neighborhood $U$ of $\theta^\ast$.

The statements for $A$ in \eqref{eq:43897589370} follow immediately from the continuity of $A$, the facts that $(D\Phi(\theta^\ast))^{-1}= O^{-1}$, $\cT_{\theta^\ast} \mathcal M = D\Phi^{-1}(\theta^\ast) (\R^{d_{\cT}}\times \{0\}^{d_{\cN}})$ and Proposition~\ref{thm:rebjock}.
\end{proof}

\section{Proof of the Main Result} \label{sec:proof}
In this section, we combine the geometric considerations of Section~\ref{sec:geometric} with the descent lemma, Lemma~\ref{prop:rate}, and Borel-Cantelli arguments to finish the proof of the main result. 

First, we use the geometric considerations of Section~\ref{sec:geometric} to simplify the statement of Lemma~\ref{prop:rate}.

\begin{lemma} \label{lem:1}
    Let $N \in \N_0$ and $\theta^\ast \in \R^d$ be a $(\mu,L, \sigma)$-regular point. Then, for every $\eps>0$ there exist a $\delta >0$, neighborhoods $\tilde U \subset U \subset \R^d$ of $\theta^\ast$, and a $C^1$-diffeomorphism $\Phi:  U \to \Phi( U)$ such that $D\Phi(\theta^\ast)$ is an orthogonal matrix, $A(\theta^\ast) := D\Phi(\theta^\ast) (\Hess f)(\theta^\ast) (D\Phi(\theta^\ast))^{-1}$ is a symmetric matrix satisfying
    \begin{equation*}
        \ker(A(\theta^\ast)) = \R^{d_{\cT}} \times \{0\}^{d_{\cN}} \quad \text{ and } \quad \spec(A(\theta^\ast) \einschraenkung_{\{0\}^{d_{\cT}} \times \R^{d_{\cN}}}) \subset [\mu, L],
    \end{equation*}
    and for all $k > N $
    \begin{align*}
    \E[\1_{\IB_k}&\|\proj_{\cN}(\Phi(\theta_k))\|^2 \mid \cF_{k-1}] \le  \1_{\IB_{k-1}}(m(\gamma)+\eps) \|\proj_{\cN}(\Phi(\theta_{k-1}))\|^2,
    \end{align*}
    where $m(\gamma)$ is defined in \eqref{eq:mgamma}, $\IB_N:=\{\theta_N \in \tilde U\}$ and, for $k>N$, 
    $$
    \IB_k := \bigcap_{i=N, \dots, k-1} \{\theta_i \in \tilde U\} \cap \bigcap_{i=N,\dots, k}\{\|D_i\| \le \delta/\gamma \}.
    $$
\end{lemma}

\begin{proof} 
For $\eps'>0$ let $U \subset \R^d$ be a neighborhood of $\theta^\ast$, $\Phi:U \to \Phi(U)$ and $C\ge 0$ be a constant such that (A1)-(A4) of Definition~\ref{ass:PL} and the conclusion of Lemma~\ref{prop:geometry} are satisfied (with $\eps$ replaced by $\eps'$).

Again, we denote $\Phi_{\cN}(\theta):=\proj_{\cN}(\Phi(\theta))$ and $A(\theta):= D\Phi(\theta) (\Hess f)(\theta) (D\Phi(\theta))^{-1}$. 
   Using (A4) of Definition~\ref{ass:PL} together with (vi) of Lemma~\ref{prop:geometry}, one has on the event $\{\theta_{k-1} \in U\}$ that 
   \begin{align*}
       \E[\|D_k\|^2 \mid \mathcal F_{k-1}] &\le \sigma (f(\theta_{k-1})-f(\theta^\ast)) \le \frac{\sigma}{2} \langle A(\theta_{k-1}) \Phi_{\cN}(\theta_{k-1}), \Phi_{\cN}(\theta_{k-1}) \rangle + \sigma \eps' \|\Phi_{\cN}(\theta_{k-1})\|^2.
   \end{align*}

Moreover, note that 
    $$
    A(\theta)= D\Phi(\theta) (\Hess f)(\theta) (D\Phi(\theta))^{-1} \to  D\Phi(\theta^\ast) (\Hess f)(\theta^\ast) (D\Phi(\theta^\ast))^{-1}= A(\theta^\ast) , 
    $$
    as $\theta \to \theta^\ast $. Note that by Lemma~\ref{prop:geometry} the matrix $D\Phi(\theta^\ast)$ is orthogonal and hence $A(\theta^\ast)$ is symmetric. Moreover, it satisfies $\spec(A(\theta^\ast) \einschraenkung_{\{0\}^{d_{\cT}} \times \R^{d_{\cN}}}) \subset [\mu, L]$.
    Therefore, 
    for a sufficiently small neighborhood $ U$ of $\theta^\ast$ one has for all $\theta \in U$
    \begin{align*}
        &\|(\Id-\gamma A(\theta)) \Phi_{\cN}(\theta)\|^2 + (1+\eps')^2 \frac{\gamma^2 \sigma}{2} \langle A(\theta) \Phi_{\cN}(\theta), \Phi_{\cN}(\theta) \rangle \\
        &\le  \|(\Id-\gamma A(\theta^\ast)) \Phi_{\cN}(\theta)\|^2 + (1+\eps')^2 \frac{\gamma^2 \sigma}{2} \langle A(\theta^\ast) \Phi_{\cN}(\theta), \Phi_{\cN}(\theta) \rangle +\eps' \|\Phi_{\cN}(\theta)\|^2 \\
        & = \langle \Big( \big(\Id-\gamma A(\theta^\ast)\big)^2 +(1+\eps')^2 \frac{\gamma^2 \sigma}{2}  A(\theta^\ast) \Big) \Phi_{\cN}(\theta) ,\Phi_{\cN}(\theta) \rangle +\eps' \|\Phi_{\cN}(\theta)\|^2 \\
        & \le \left(\sup_{\lambda \in \spec(A(\theta^\ast) \einschraenkung_{\{0\}^{d_{\cT}} \times \R^{d_{\cN}}}) }(1 - \gamma \lambda)^2+(1+\eps')^2\frac{\gamma^2\sigma \lambda}{2} +\eps' \right)\|\Phi_{\cN}(\theta)\|^2\\
    & \le \max\left((1 - \gamma \mu)^2+(1+\eps')^2\frac{\gamma^2\sigma \mu}{2}+\eps', (1 - \gamma L)^2+(1+\eps')^2\frac{\gamma^2\sigma L}{2}+\eps'\right)\|\Phi_{\cN}(\theta)\|^2.
    \end{align*}
 
    Now, using Lemma~\ref{prop:rate} together with the bounds given by Lemma~\ref{prop:geometry}, there exist a constant $C(\eps')\ge 0$ with $C(\eps')\to 0$ as $\eps'\to 0$, a constant $\tilde C(\eps')\ge 0$, a $\delta(\eps')>0$ and neighborhoods $\tilde U_{\eps'} \subset U_{\eps'}$ of $\theta^\ast$ such that for all $k > N$
    $$
       \E[\1_{\IB_k}\|\Phi_{\cN}(\theta_k)\|^2 \mid \cF_{k-1}] \le  \1_{\IB_{k-1}} 
       \Big(m(\gamma) + C(\eps') +\tilde C (\eps') \frac{R}{\delta(\eps')}\Big)
       \|\Phi_{\cN}(\theta_{k-1})\|^2,
    $$
    where $R:= \sup_{\theta \in \tilde U_{\eps'}}\|\proj_{\cN}(\Phi(\theta))\|$ and $\tilde U_{\eps'}>0$ and $\delta(\eps')$ are chosen so that $\bigcup_{\theta \in \tilde U_{\eps'}} \overline{B_{2\delta(\eps')}(\theta)} \subset U_{\eps'}$ and $\|\nabla f\|_{\infty,\tilde U_{\eps'}}\le \delta(\eps')/\gamma$. Now, fix $\eps'>0$ such that $C(\eps') \le \frac{\eps}{2}$. Then, one can shrink the set $\tilde U_{\eps'}$ so that $R=\sup_{\theta \in \tilde U_{\eps'}}\|\proj_{\cN}(\Phi(\theta))\|\le (\eps \delta(\eps'))/(2 \tilde C(\eps'))$ and, thus,
    $$
       \E[\1_{\IB_k}\|\Phi_{\cN}(\theta_k)\|^2 \mid \cF_{k-1}] \le  \1_{\IB_{k-1}} 
       (m(\gamma) + \eps)
       \|\Phi_{\cN}(\theta_{k-1})\|^2,
    $$
    where
    \begin{equation*}
    m(\gamma):=\max\left((1 - \gamma \mu)^2+\frac{\gamma^2\sigma \mu}{2}, (1 - \gamma L)^2+\frac{\gamma^2\sigma L}{2}\right).
\end{equation*}
\end{proof}

Next, we show that, once SGD enters a sufficiently small neighborhood of a $(\mu,L,\sigma)$-regular point, it converges with high probability to a $(\mu,L,\sigma)$-regular point with the same objective function value.

\begin{lemma} \label{lem:2}
    Let $\theta^\ast \in \R^d$ be a $(\mu,L, \sigma)$-regular point and assume that
    \begin{equation}\label{eq:step_size_lem:2}
        0<\gamma < \frac{2}{L+\frac{\sigma }{2}}.
    \end{equation}
    Then, for every $\delta >0$ there exists an open neighborhood $V$ of $\theta^\ast$ such that for the event $\IV := \{\exists k \in \N_0 : \theta_k \in V \}$ the following holds: 
    there exists a compact set $K$ of $(\mu,L, \sigma)$-regular points and an event $\IC$ with $\P(\IC \cap \IV) \ge (1-\delta) \P(\IV)$ such that, on $\IC \cap \IV$, the SGD scheme $(\theta_k)_{k \in \N_0}$ almost surely converges to a (random) $\theta_\infty$ in $K$ with $f(\theta_\infty) = f(\theta^\ast)$ and, for every $\kappa >0$, one has $\|D_k\| \le \kappa $ for all but finitely many $k$'s.
\end{lemma}

\begin{proof}
    For $k \in \N_0$ and a neighborhood $V$ of $\theta^\ast$ denote 
    $$
    \IV_k:=\{\theta_k\in V\} \cap \bigcap_{i=0}^{k-1} \{\theta_i\notin V\}.
    $$
    In the following, we will prove that, for each $N\in\N_0$, there exists an event $\IC_N$ such that
    $$
    \P(\IC_N\cap\IV_N)\ge (1-\delta)\P(\IV_N)
    $$
    and such that the asserted convergence properties hold on $\IC_N\cap\IV_N$. Then, since $\IV = \dot \bigcup_{N \in \N_0} \IV_N$, for the event $\IC:=\bigcup_{N=0}^\infty(\IC_N\cap\IV_N)$ it holds that
    $$
        \P(\IC \cap \IV) = \sum_{N \in \N_0} \P(\IC_N \cap \IV_N) \ge (1-\delta) \sum_{N \in \N_0} \P(\IV_N) = (1-\delta) \P(\IV).
    $$
    
    The restriction \eqref{eq:step_size_lem:2} on the step-size $\gamma$ guarantees that $m(\gamma ) <1$ (see Section~\ref{sec:quadratic}). Thus, by Lemma~\ref{lem:1} there exist constants $M \ge 0$, $m' \in (0,1)$, neighborhoods $\tilde U \subset U $ of $\theta^\ast$, and a $C^1$-diffeomorphism $\Phi:  U \to \Phi( U)$ such that for all $N \in \N_0$ and $k > N$
    \begin{align}\label{eq:78977869}
    \E[\1_{\IB_k}&\|\proj_{\cN}(\Phi(\theta_k))\|^2 \mid \cF_{k-1}] \le  \1_{\IB_{k-1}}m' \,  \|\proj_{\cN}(\Phi(\theta_{k-1}))\|^2,
    \end{align}
    where $\IB_N:= \{\theta_N \in \tilde U\}$ and, for $k > N$, $\IB_k := \bigcap_{i=N, \dots, k-1} \{\theta_i \in \tilde U\} \cap \bigcap_{i=N,\dots, k}\{\|D_i\| \le M \}$. We bound the expected length of the trajectory after possibly shrinking $\tilde U$. Using Lemma~\ref{prop:geometry} (v) and (vi), for a sufficiently small neighborhood $\tilde U$ there exists a constant $C \ge 0$ such that for all $\theta \in \tilde U$, $\|\nabla f(\theta)\| \le \frac{C}{2}\|\proj_{\cN}(\Phi(\theta))\|$, and for all $k \in \N$ one has on the event $\{\theta_{k-1} \in \tilde U\}$ that 
    \begin{align} \label{eq:23798762}
         \E[\|D_k\|^2 \mid \mathcal F_{k-1}]^{1/2} \le (\sigma (f(\theta_{k-1})-f(\theta^\ast)))^{1/2}\le \frac{C}{2} \|\proj_{\cN}(\Phi(\theta_{k-1}))\|.
    \end{align}
    Thus, 
    \begin{align} \begin{split} \label{eq:23894729862}
        \sum_{k=N+1}^\infty \E[\1_{\IB_k} \|\theta_k-\theta_{k-1}\| \mid \cF_N] & \le \sum_{k=N+1}^\infty \gamma \E[\1_{\IB_{k-1}} (\|\nabla f(\theta_{k-1})\| + \|D_k\|) \mid \cF_N] \\
        & \le \sum_{k=N+1}^\infty \gamma C \E[\1_{\IB_{k-1}} \|\proj_{\cN}(\Phi(\theta_{k-1}))\|^2 \mid \cF_N]^{1/2} \\
        & \le \sum_{k=N+1}^\infty \1_{\IB_N} \gamma C (m')^{(k-N-1)/2} R < \infty ,
        \end{split}
    \end{align}
    where $R:=\sup_{\theta\in \tilde U}\|\proj_{\cN}(\Phi(\theta))\|$ and in the last inequality we have used \eqref{eq:78977869}. On the event $\IC_N := \bigcap_{k \ge N}\IB_k$ this implies almost sure convergence of $(\theta_k)_{k \in \N}$ and, since $\tilde U$ is bounded, the limit $\theta_\infty := \lim_{k \to \infty} \theta_k$ is almost surely inside a compact set $K$.
    Since $\|\proj_{\cN}(\Phi(\theta_k))\| \to 0$, by Lemma~\ref{prop:geometry} we get $\nabla f(\theta_\infty)=0$ and $f(\theta_\infty) = f(\theta^\ast)$, on $\IC_N$. 
    
    Next, we  show that for every $\delta>0$ there exists a neighborhood $V$ of $\theta^\ast$ such that, for all $N \in \N_0$, $\P(\IC_N \cap  \IV_N) \ge (1-\delta) \P(\IV_N)$.
    
    Let $r_1>0$ be such that $B_{r_1}(\theta^\ast) \subset \tilde U$. Then for $r_2>0$ with $\Phi^{-1}(B_{r_2}(\Phi(\theta^\ast))) \subset B_{r_1/2}(\theta^\ast)$, we set $V:=\Phi^{-1}(B_{r_2}(\Phi(\theta^\ast)))$. For $N \in \N_0$ and $k > N$ analogously to \eqref{eq:23894729862} one has
    \begin{equation}\label{eq:478397598}
        \sum_{i=N+1}^k \E[\1_{\IB_i\cap \IV_N} \|\theta_i-\theta_{i-1}\| \mid \cF_N] \le \1_{\IV_N} \sum_{i=N+1}^k \gamma C (m')^{(i-N-1)/2} r_2.
    \end{equation}   

    Note that, since
    $$
        \IC_N^c \cap \IV_N \subset \Bigg( \bigcup_{k=N}^\infty \Big( \IB_k \cap \IV_N\cap \{\theta_k\notin \tilde U\} \Big) \Bigg)
        \cup
        \Bigg( \bigcup_{k=N}^\infty \Big( \IB_k\cap \IV_N\cap \{\|D_{k+1}\|>M\} \Big) \Bigg),
    $$
    we have
    \begin{equation}\label{eq:347902}
        \P(\IC_N^c \cap \IV_N)
        \le \P\Bigg( \Bigg( \bigcup_{k=N}^\infty \Big( \IB_k \cap \IV_N\cap \{\theta_k\notin \tilde U\} \Big) \Bigg) \Bigg) 
        + \sum_{k=N}^\infty \E[ \1_{\IB_k\cap \IV_N} \1_{\{\|D_{k+1}\|>M\}}].
    \end{equation}
    Using the triangle inequality, we have on $\IV_N$
    $$
        \|\theta_k-\theta^\ast\| \le \|\theta_k-\theta_N\|+\|\theta_N-\theta^\ast\| \le \|\theta_k-\theta_N\|+\frac{r_1}{2}.
    $$
    Combining this with the assumption that
    $B_{r_1}(\theta^\ast)\subset \tilde U$ implies that, on $\IV_N$, $\{\theta_k\notin \tilde U\}\subset\{\|\theta_k-\theta_N\|\ge \frac{r_1}{2}\}$ and, therefore,
    $$
        \bigcup_{k=N}^\infty \Big( \IB_k \cap \IV_N\cap \{\theta_k\notin \tilde U\} \Big)  \subset \IV_N \cap \Bigg\{ \sum_{i=N+1}^\infty \1_{\IB_i}\|\theta_i-\theta_{i-1}\| \ge \frac{r_1}{2} \Bigg\}.
    $$
    Thus, by Markov's inequality and
    \eqref{eq:478397598},
    \begin{align}
    \begin{split}
        \P\Bigg( \bigcup_{k=N}^\infty \Big( \IB_k \cap \IV_N\cap \{\theta_k\notin \tilde U\} \Big) \,\Bigg|\, \cF_N \Bigg)
        &\le \frac{2}{r_1}
        \sum_{i=N+1}^\infty \E[\1_{\IB_i\cap \IV_N}\|\theta_i-\theta_{i-1}\|\mid \cF_N] \\
        &\le \1_{\IV_N}\frac{2\gamma C r_2}{r_1}\sum_{i=N+1}^\infty (m')^{(i-N-1)/2}\\
        &=\1_{\IV_N} \frac{2\gamma C r_2}{r_1(1-\sqrt{m'})}.
    \end{split}
    \label{eq:3458497529100}
    \end{align}
    Similarly, by Markov's inequality, \eqref{eq:78977869}, and \eqref{eq:23798762},
    \begin{align} \label{eq:2304797222}
    \begin{split}
        \sum_{k=N}^\infty \E[\1_{\IB_k \cap \IV_N \cap \{\|D_{k+1}\| > M\}}\mid \cF_N]
        &\le \sum_{k=N}^\infty
        \frac{1}{M} \E[ \1_{\IB_k \cap \IV_N}\|D_{k+1}\| \mid \cF_N] \\
        &\le \sum_{k=N}^\infty \frac{C}{2 M} \E[ \1_{\IB_k \cap \IV_N} \|\proj_{\cN}(\Phi(\theta_k))\|^2\mid \cF_N]^{1/2} \\
        &\le \sum_{k=N}^\infty \1_{\IV_N} \frac{C}{2M} (m')^{(k-N)/2} r_2 \le \1_{\IV_N} \frac{C r_2}{2M(1-\sqrt{m'})} .
    \end{split}
    \end{align}

    Combining \eqref{eq:347902}, \eqref{eq:3458497529100}, and
    \eqref{eq:2304797222}, we obtain
    $$
        \P(\IC_N^c\cap \IV_N\mid \cF_N) \le \1_{\IV_N}\frac{r_2}{1-\sqrt{m'}}
        \Big( \frac{2\gamma C}{r_1} +\frac{C}{2M} \Big),
    $$
    so that for a sufficiently small $r_2>0$ we have
    $ \P(\IC_N^c\cap \IV_N\mid \cF_N) \le \delta \1_{\IV_N}$. Taking expectation gives
    $$
        \P(\IC_N\cap \IV_N) =\P(\IV_N)-\P(\IC_N^c\cap \IV_N)\ge(1-\delta)\P(\IV_N).
    $$

    Lastly, note that for every $\kappa>0$ one has almost surely on $\IC_N\cap\IV_N$ that $\|D_k\|\le \kappa$ for all but finitely many $k$'s. Indeed, applying \eqref{eq:2304797222} with $M$ replaced by $\kappa$ gives
    \begin{align*}
        \sum_{k = N}^\infty \P(\IC_N \cap \IV_N \cap \{\|D_{k+1}\| >\kappa \}) & \le \sum_{k=N}^\infty \E[\1_{\IB_k \cap \IV_N \cap \{\|D_{k+1}\| > \kappa\}}] < \infty,
    \end{align*}
    so that the statement follows from the Borel-Cantelli lemma.
\end{proof}

We are now in a position to finish the proof of the main result.

\begin{proof}[Proof of Theorem~\ref{theo1}]
    Let $\delta>0$. By Lemma~\ref{lem:2}, there exists an open neighborhood $V$ of $\theta^\ast$, a compact set $K$ of $(\mu,L, \sigma)$-regular points  and an event $\IC$ with $\P(\IC \cap \IV) \ge (1-\delta) \P(\IV)$ such that, on $\IC \cap \IV$, the SGD scheme $(\theta_k)_{k \in \N_0}$ almost surely converges to a (random) point $\theta_\infty$ in $K$ with $f(\theta_\infty) = f(\theta^\ast)$, where $\IV := \{\exists k \in \N_0 : \theta_k \in V\}$.

    It remains to show that \eqref{eq:rate1} and \eqref{eq:mgamma} holds. 
    Fix $\eps >0$ such that $m(\gamma) +\eps < 1$. Using Lemma~\ref{lem:1}, for every $(\mu,L, \sigma)$-regular point $\hat \theta$ in $K$ there exist a constant $M >0$, neighborhoods $\tilde U \subset U \subset \R^d$ of $\hat \theta$, and a $C^1$-diffeomorphism $\Phi:  U \to \Phi( U)$ such that $D\Phi(\hat \theta)$ is an orthogonal matrix and for all $N \in \N_0$ and  $k > N$
    \begin{align} \label{eq:2349857293}
    \E[\1_{\IB_k}&\|\proj_{\cN}(\Phi(\theta_k))\|^2 \mid \cF_{k-1}] \le  \1_{\IB_{k-1}}(m(\gamma)+\eps/2) \|\proj_{\cN}(\Phi(\theta_{k-1}))\|^2,
    \end{align}
    where $\IB_N:= \{\theta_N \in \tilde U\}$ and, for $k > N$, $\IB_k := \bigcap_{i=N, \dots, k-1} \{\theta_i \in \tilde U\} \cap \bigcap_{i=N,\dots, k}\{\|D_i\| \le M \}$. 
    Similarly to \eqref{eq:23894729862}, after possibly shrinking the set $\tilde U$ this implies the existence of a constant $C_1\ge 0$ such that for every $k \ge N$
    \begin{align*}
        \E[\1_{\IB_\infty}  \|\theta_k-\theta_\infty\| \mid \mathcal F_N] &\le \sum_{j=k+1}^\infty \E[\1_{\IB_j}  \|\theta_j-\theta_{j-1}\| \mid \mathcal F_N]  \\
        &\le \1_{\IB_N}  C_1 \Big(m(\gamma)+\frac {\eps}{2}\Big)^{(k-N)/2},
    \end{align*}
    where $\IB_\infty := \cap_{k \ge N} \IB_k$. Using Borel-Cantelli together with
    \begin{align} \label{eq:2543858734941}
        \P(\IB_\infty \cap \{\|\theta_k-\theta_\infty\| \ge (m(\gamma)+\sfrac {3\eps}{4})^{(k-N)/2} \}) \le C_1 \Big(\frac{m(\gamma)+\frac {\eps}{2}}{m(\gamma)+\frac {3\eps}{4}}\Big)^{(k-N)/2}
    \end{align}
    we conclude that, on $\IB_\infty$, \eqref{eq:rate1} holds almost surely.
    
    Since there exists a finite open cover $\tilde U_1, \dots, \tilde U_\ell$ of $K$ consisting of neighborhoods that satisfy \eqref{eq:2543858734941}, \eqref{eq:rate1} holds almost surely on
    $$
        \IC \cap \IV \subset \bigcup_{i=1, \dots, \ell} \bigcup_{N \in \N_0}   \bigcap_{k \ge N} \{\theta_k \in \tilde U_i, \|D_k\| \le M_i\},
    $$
    where $M_1, \dots, M_\ell >0$ denote the constants in the definition of the events $\IB_k$ in \eqref{eq:2349857293} for the respective neighborhood $\tilde U_1, \dots, \tilde U_\ell$ and we have used that, on $\IC\cap\IV$, one has $\theta_k\to\theta_\infty\in K$, and for every $\kappa>0$ one has $\|D_k\|\le\kappa$ for all but finitely many $k$'s.
    \eqref{eq:mgamma} now follows from
    $$
        f(\theta)-f(\hat \theta) \le \frac 12 L \|\theta-\hat \theta\|^2 + o(\|\theta -\hat \theta\|^2),
    $$
    for every $(\mu,L, \sigma)$-regular point $\hat \theta$.
\end{proof}

\section*{Acknowledgments} SK acknowledges funding by the Deutsche Forschungsgemeinschaft (DFG, German Research Foundation) – CRC/TRR 388 ``Rough Analysis, Stochastic Dynamics and Related Fields'' – Project ID 516748464.
TK acknowledges funding by the DFG - CRC 1701 ``Port-Hamiltonian Systems'' – Project ID 531152215. 
        
\bibliographystyle{plain}
\bibliography{Asymptotic_GD_with_PL}

@incollection{benaim2006dynamics,
  title={Dynamics of stochastic approximation algorithms},
  author={Bena{\"\i}m, Michel},
  booktitle={Seminaire de probabilites XXXIII},
  pages={1--68},
  year={2006},
  publisher={Springer}
}

@inproceedings{karimi2016linear,
  title={Linear convergence of gradient and proximal-gradient methods under the \uppercase{P}olyak-\uppercase{{\L}}ojasiewicz condition},
  author={Karimi, H. and Nutini, J. and Schmidt, M.},
  booktitle={Joint European Conference on Machine Learning and Knowledge Discovery in Databases},
  pages={795--811},
  year={2016},
  organization={Springer}
}

@article{lojasiewicz1963propriete,
  title={Une propri{\'e}t{\'e} topologique des sous-ensembles analytiques r{\'e}els},
  author={{\L}ojasiewicz, S.},
  journal={Les {\'e}quations aux d{\'e}riv{\'e}es partielles},
  volume={117},
  pages={87--89},
  year={1963}
}

@article{polyak1963gradient,
  title={Gradient methods for the minimisation of functionals},
  author={Polyak, B. T.},
  journal={U.S.S.R. Comput. Math. and Math. Phys.},
  volume={3},
  number={4},
  pages={864--878},
  year={1963},
  publisher={Elsevier}
}

@article{polyak1964some,
  title={Some methods of speeding up the convergence of iteration methods},
  author={Polyak, B. T.},
  journal={U.S.S.R. Comput. Math. and Math. Phys.},
  volume={4},
  number={5},
  pages={1--17},
  year={1964},
  publisher={Elsevier}
}

@article{fehrman2020convergence,
  title={Convergence rates for the stochastic gradient descent method for non-convex objective functions},
  author={Fehrman, B. and Gess, B. and Jentzen, A.},
  journal={J. Mach. Learn. Res.},
  volume={21},
  number={136},
  pages={1--48},
  year={2020}
}

@article{rebjock2023fast,
  title={Fast convergence to non-isolated minima: four equivalent conditions for {$C^2$} functions},
  author={Rebjock, Quentin and Boumal, Nicolas},
  journal={Math. Program.},
    volume={213},
  number={1},
  pages={151--199},
  year={2025},
}

@article{dereich2023central,
  title={Central limit theorems for stochastic gradient descent with averaging for stable manifolds},
  author={Dereich, Steffen and Kassing, Sebastian},
  journal={Electron. J. Probab.},
  volume={28},
  pages={1--48},
  year={2023},
  publisher={The Institute of Mathematical Statistics and the Bernoulli Society}
}

@article{wojtowytsch2021stochastic,
  title={Stochastic gradient descent with noise of machine learning type part {I}: Discrete time analysis},
  author={Wojtowytsch, S.},
  journal={J. Nonlinear Sci.},
  volume={33},
  number={3},
  pages={45},
  year={2023},
  publisher={Springer}
}

@article{dereich2021convergence,
  title={Convergence of stochastic gradient descent schemes for \uppercase{{\L}}ojasiewicz-landscapes},
  author={Dereich, Steffen and Kassing, Sebastian},
  journal={J. Mach. Learn.},
  volume = {3},
  number = {3},
  pages = {245--281},
  year={2024}
}

@article{feehan2020morse,
  title={On the {M}orse--{B}ott property of analytic functions on Banach spaces with {{\L}}ojasiewicz exponent one half},
  author={Feehan, Paul},
  journal={Calc. Var. Partial Differential Equations},
  volume={59},
  number={2},
  pages={87},
  year={2020},
  publisher={Springer}
}

@inproceedings{vaswani2019fast,
  title={Fast and faster convergence of {S}{G}{D} for over-parameterized models and an accelerated perceptron},
  author={Vaswani, S. and Bach, F. and Schmidt, M.},
  booktitle={The 22nd international conference on artificial intelligence and statistics},
  pages={1195--1204},
  year={2019},
  organization={PMLR}
}

@article{khaledbetter,
  title={Better Theory for {S}{G}{D} in the Nonconvex World},
  author={Khaled, A. and Richt{\'a}rik, P.},
  journal={Transactions on Machine Learning Research},
year = {2023}
}

@article{robbins1951stochastic,
  title={A stochastic approximation method},
  author={Robbins, Herbert and Monro, Sutton},
  journal={The annals of mathematical statistics},
  pages={400--407},
  year={1951},
  publisher={JSTOR}
}

@article{absil2005convergence,
  title={Convergence of the iterates of descent methods for analytic cost functions},
  author={Absil, Pierre-Antoine and Mahony, Robert and Andrews, Ben},
  journal={SIAM Journal on Optimization},
  volume={16},
  number={2},
  pages={531--547},
  year={2005},
  publisher={SIAM}
}

@article{garrigos2023handbook,
  title={Handbook of convergence theorems for (stochastic) gradient methods},
  author={Garrigos, Guillaume and Gower, Robert M},
  journal={arXiv preprint arXiv:2301.11235},
  year={2023}
}

@article{gess2026exponential,
  title={Exponential convergence rates for momentum stochastic gradient descent in the overparametrized setting},
  author={Gess, Benjamin and Kassing, Sebastian},
  journal={Mathematical Programming},
  pages={1--37},
  year={2026},
  publisher={Springer}
}

@article{attouch2013convergence,
  title={Convergence of descent methods for semi-algebraic and tame problems: proximal algorithms, forward--backward splitting, and regularized {G}auss--{S}eidel methods},
  author={Attouch, Hedy and Bolte, J{\'e}r{\^o}me and Svaiter, Benar Fux},
  journal={Mathematical Programming},
  volume={137},
  number={1},
  pages={91--129},
  year={2013},
  publisher={Springer}
}

@article{ochs2018local,
  title     = {Local Convergence of the {Heavy-Ball} Method and {iPiano} for {Non-convex} Optimization},
  author    = {Ochs, Peter},
  journal   = {Journal of Optimization Theory and Applications},
  volume    = {177},
  pages     = {153--180},
  year      = {2018},
  doi       = {10.1007/s10957-018-1272-y},
  url       = {https://doi.org/10.1007/s10957-018-1272-y},
  publisher = {Springer}
}

@article{kassing2024polyak,
  title={Polyak's Heavy Ball Method Achieves Accelerated Local Rate of Convergence under {P}olyak-{L}ojasiewicz Inequality},
  author={Kassing, Sebastian and Weissmann, Simon},
  journal={arXiv preprint arXiv:2410.16849},
  year={2024}
}

@inproceedings{gupta2024nesterov,
  title={Nesterov acceleration in benignly non-convex landscapes},
  author={Gupta, Kanan and Wojtowytsch, Stephan},
  booktitle={International Conference on Learning Representations},
  volume={2025},
  pages={18741--18773},
  year={2025}
}

@book{kushner2003stochastic,
  title={Stochastic approximation and recursive algorithms and applications},
  author={Kushner, Harold J and Yin, G George},
  year={2003},
  publisher={Springer}
}

@unpublished{nejma2025polyak,
  title={Polyak-{{\L}}ojasiewicz inequality is essentially no more general than strong convexity for ${C}^2$ functions},
  author={Nejma, Aziz Ben},
  note={arXiv preprint arXiv:2512.05285},
  year={2025}
}

@misc{criscitiello2025,
  author = {Criscitiello, Chris and Rebjock, Quentin and Boumal,
    Nicolas},
  title = {If a Smooth Function Is Globally {PŁ} and Coercive, Then It
    Has a Unique Minimizer},
    year = {2025},
  date = {2025-02-08},
  howpublished = {\url{https://www.racetothebottom.xyz/posts/PL-smooth-unique/}},
  langid = {en}
}

@article{tadic2015convergence,
  title={Convergence and convergence rate of stochastic gradient search in the case of multiple and non-isolated extrema},
  author={Tadi{\'c}, Vladislav B},
  journal={Stochastic Processes and their Applications},
  volume={125},
  number={5},
  pages={1715--1755},
  year={2015},
  publisher={Elsevier}
}

@book{ljung1992stochastic,
  title={Stochastic approximation and optimization of random systems},
  author={Ljung, Lennart and Pflug, Georg Ch and Walk, Harro},
  volume={17},
  year={1992},
  publisher={Springer Science \& Business Media}
}

@article{mertikopoulos2020almost,
  title={On the almost sure convergence of stochastic gradient descent in non-convex problems},
  author={Mertikopoulos, Panayotis and Hallak, Nadav and Kavis, Ali and Cevher, Volkan},
  journal={Advances in Neural Information Processing Systems},
  volume={33},
  pages={1117--1128},
  year={2020}
}

@article{bottou2018optimization,
  title={Optimization methods for large-scale machine learning},
  author={Bottou, L{\'e}on and Curtis, Frank E and Nocedal, Jorge},
  journal={SIAM review},
  volume={60},
  number={2},
  pages={223--311},
  year={2018},
  publisher={SIAM}
}

@inproceedings{ma2018power,
  title={The power of interpolation: Understanding the effectiveness of {S}{G}{D} in modern over-parametrized learning},
  author={Ma, Siyuan and Bassily, Raef and Belkin, Mikhail},
  booktitle={International Conference on Machine Learning},
  pages={3325--3334},
  year={2018},
  organization={PMLR}
}

@article{liu2023aiming,
  title={Aiming towards the minimizers: fast convergence of {S}{G}{D} for overparametrized problems},
  author={Liu, Chaoyue and Drusvyatskiy, Dmitriy and Belkin, Misha and Davis, Damek and Ma, Yian},
  journal={Advances in neural information processing systems},
  volume={36},
  pages={60748--60767},
  year={2023}
}

@article{feng2026optimal,
  title={Optimal local linear convergence of Nesterov's accelerated gradient method for $ C^2$ functions under the {P}olyak--{{\L}}ojasiewicz inequality},
  author={Feng, Zixu and Yuan, Hao},
  journal={arXiv preprint arXiv:2603.21516},
  year={2026}
}
		
	\end{document}